\documentclass[12pt]{amsart}

\usepackage[utf8]{inputenc}
\usepackage[]{fontenc}
\usepackage{hyperref}
\usepackage{amssymb}
\usepackage{amsthm}
\usepackage{graphicx}
\usepackage{import}
\usepackage{amsmath}
\usepackage{amscd}
\usepackage{mathtools}
\usepackage{verbatim}
\usepackage[normalem]{ulem}
\usepackage{tikz-cd}
\usepackage{color}
\usepackage[all]{xy}
\usepackage[margin = 3.5cm, marginpar = 2.5cm]{geometry}
\usepackage{todonotes}
\usepackage{verbatim}
\usepackage{wasysym, stackengine, makebox, graphicx}
\usepackage{pinlabel}

\theoremstyle{theorem}
\newtheorem{thm}{Theorem}[section]
\newtheorem{prop}[thm]{Proposition}
\newtheorem{lem}[thm]{Lemma}
\newtheorem{cor}[thm]{Corollary}
\numberwithin{equation}{section}

\theoremstyle{definition}

\newtheorem*{question*}{Question}
\newtheorem{defi}[thm]{Definition}

\newtheorem{lemma}[thm]{Lemma}

\theoremstyle{remark}
\newtheorem{rk}[thm]{Remark}
\newtheorem{example}[thm]{Example}

\renewcommand{\P}{\mathbb{P}}
\newcommand{\CP}{\mathbb{CP}^2}
\newcommand{\CPone}{\mathbb{CP}^1}
\newcommand{\FS}{{\rm FS}}
\newcommand{\C}{\mathbb{C}}

\newcommand{\sC}{\mathcal{C}}
\newcommand{\Cns}{\mathcal{C}^{\rm ns}}

\newcommand{\Z}{\mathbb{Z}}
\newcommand{\Q}{\mathbb{Q}}
\newcommand{\sD}{\mathcal{D}}
\newcommand{\hD}{\mathbb{D}}

\renewcommand{\th}{^{\rm th}}
\newcommand{\del}{\partial}
\newcommand{\co}{\colon}
\newcommand{\inv}{^{-1}}
\renewcommand{\epsilon}{\varepsilon}

\newcommand{\col}{{\rm col}}

\newcommand{\MG}[1]{\marginpar{\tiny \color{red} {\bf M:} {#1}}}

\newcommand{\Hadd}[1]{{\color{blue} {\bf H:} #1}}

\DeclareMathOperator{\im}{im}
\DeclareMathOperator{\Int}{Int}
\DeclareMathOperator{\Sing}{Sing}

\DeclareMathOperator{\ord}{ord}

\title[Oka and Alexander polynomials of symplectic curves]{Oka and Alexander polynomials of symplectic curves and divisibility relations}

\author{Hanine Awada}
\email{hawada@bcamath.org}
\address{Basque Centre for Applied Mathematics, Bilbao, Spain}

\author{Marco Golla}
\email{marco.golla@univ-nantes.fr}
\address{CNRS et Nantes Universit\'e, Laboratoire de Math\'ematiques Jean Leray, Nantes, France}

\begin{document}

\maketitle

\begin{abstract}
We prove Libgober's divisibility relations for Oka and Alexander polynomials of symplectic curves in the complex projective plane. Along the way, we give new proofs of the divisibility relations for the Alexander polynomials of complex algebraic curves with respect to a generic line at infinity.
\end{abstract}

\section{Introduction}
The study of fundamental groups of complements of plane complex curves (and, more generally, of complex projective hypersurfaces) has a very rich history, dating back at least to the work of Zariski and van Kampen ~\cite{zar1, zar2}.
Their \emph{pencil section method} allows to find a presentation of $\pi_1(\CP \setminus \sC)$ for a complex curve $\mathcal{C}$. In the early 80s, Moishezon~\cite{Mo} introduced the notion of \emph{braid monodromy} to give a similar presentation of the fundamental group; see also Libgober~\cite{Lib1}, Cohen and Suciu~\cite{Cohen} for further advances and for more references.

Distinguishing groups directly from a presentation is hard.
Like in classical knot theory, one can use the \emph{Alexander polynomial} as a proxy to study the fundamental group (see, for istance,~\cite[Section~7.D]{Rolfsen} or~\cite[Chapter~11]{Lickorish}). 
In the context of complex curves in $\CP$ (and, more generally, of projective hypersurfaces), the Alexander polynomial was introduced and studied by Libgober in~\cite{Libcyclic}. 
To define it, one needs to choose an auxiliary line $L$. We will denote it with $\Delta_{\sC,L}$, but we will often drop $L$ from the notation when it is understood.  Later on, \emph{Oka polynomials} were also introduced for curves  (see \cite{ACT-survey, Oka} for some generalities on Oka polynomials).

Moreover, Libgober discovered a beautiful relation, not at all apparent from the definition, between the Alexander polynomial of a curve, the local Alexander polynomials of the links of its singularities, and the Alexander polynomial of the \emph{link at infinity} of $\sC$ (with respect to $L$).  Same relations hold for Oka polynomials of plane algebraic curves (see \cite[Thm 2.10.]{ACT-survey}, \cite{CogolludoFlorens}).

\subsection{Main results}
We will prove that Libgober's divisibility relations hold also for symplectic curves. Here we adopt the terminology of~\cite{golla_starkston_2022}: a \emph{symplectic curve} is a (possibly singular) 2-dimensional submanifold of $\CP$ that is $J$-holomorphic for some almost-complex structure $J$ compatible with the Fubini--Study metric $\omega_\FS$ on $\CP$. For instance, (reduced) complex curves are by definition symplectic, and so are smoothly embedded symplectic surfaces of $\CP$.\\
Let $\sC \cup L$ be a symplectic curve. Call $K_1, \dots, K_\nu$ the links of its singularities, and $K_\infty$ the link at infinity of $\sC$ with respect to $L$ . We prove the following:
\begin{thm}\label{t:main1}
The Alexander polynomial  $\Delta_{\sC,L}$ of $\sC$ relative to $L$ satisfies the following divisibility relations: 
\begin{itemize}
\item[(i)] $\Delta_{\sC,L} \mid \Delta_{K_\infty}$;
\item[(ii)] $\Delta_{\sC,L} \mid \hat\Delta_{K_1}\cdots\hat\Delta_{K_{\nu}}\cdot (1-t)^{b_1(\sC \cup L)}$; if $\sC$ is irreducibile, one can get rid of the factors $1-t$ on the right-hand side.
\end{itemize}
\end{thm}

Here, $\hat\Delta_{K_i}$ is a suitably twisted version of the classical Alexander polynomial of $K_i$, which can equivalently be described as a specific evaluation of the multi-variable Alexander polynomial of $K_i$.
Whenever $K_i$ is the link of a singularity that is \emph{not} on the line at infinity $L$, $\hat\Delta_{K_i}$ is just the usual (one-variable) Alexander polynomial of $K_i$.
\begin{thm}\label{t:main}
 Let $\phi: \pi_1(\CP\setminus(\sC \cup L))\to \Z$ be an epimorphism sending each meridian of a component $\sC_j$ of degree $d_j$ of $\sC$ to a non-zero integer $a_j$, such that $\sum_j d_j a_j \ne 0$. Denote by  $\Delta^{\phi}_{\sC,L}$ the Oka polynomial  of $\sC$ with respect to $L$, then
\begin{itemize}
\item[(i)] $\Delta^{\phi}_{\sC,L} \mid \Delta^{\phi}_{K_\infty}$;
\item[(ii)] $\Delta^{\phi}_{\sC,L} \mid \hat\Delta_{K_1}\cdots\hat\Delta_{K_{\nu}}\cdot \prod_j(1-t^{a_j})^{-\chi(\Cns_j)+2}$;
\end{itemize}
where $\chi(\Cns_j)$ is the Euler characteristic of the non-singular part of $\sC_j$.\\
Moreover, if $\sC$ is a union of rational curves, then the second divisibility becomes $$\Delta^{\phi}_{\sC,L} \mid \hat\Delta_{K_1}\cdots\hat\Delta_{K_{\nu}}\cdot\prod_j (1-t^{a_j})^{-\chi(\Cns_j)+1}.$$ 
\end{thm}

In particular, Theorem~\ref{t:main1} recovers Libgober's result for complex curves. We would like to stress that our proofs are not just an adaptation of Libgober's proofs. In fact, we also give alternative proofs in the case of complex curves (at least under the assumption that $L$ is generic) that have a more low-dimensional flavour than the original arguments.

As a corollary, on the one hand we get strong restrictions on the polynomials that can arise as Alexander/Oka polynomials of symplectic curves, and on other we recover a folklore result (attributed to Zariski, and related to~\cite{zar3}) on Alexander polynomials of irreducible complex curves of prime power degree and we extend it to the symplectic setup.

\begin{cor}\label{c:cyclotomic}
The Alexander and Oka polynomial of a symplectic curve are $\CP$ is either $0$ or a product of cyclotomic polynomials.
\end{cor}

\begin{rk}\label{r:cyclautomatic}
For instance, $\Delta_{\sC,L}$ (or $\Delta^{\phi}_{\sC,L}$)  does not vanish, and hence is a product of cyclotomic polynomials, in the following cases:
\begin{itemize}
\item if $L$ is transverse to $\sC$;
\item if $\sC$ is a line arrangement, except for case where $\sC \cup L$ is a pencil;
\item if $\sC$ is irreducible.
\end{itemize}
\end{rk}

\begin{cor}\label{c:primepowerdegreeirreducible}
Let $\sC \subset \CP$ be an irreducible symplectic curve whose order is a prime power.
Then the Alexander polynomial of $\sC$ with respect to a generic line at infinity is $1$.
\end{cor}

\subsection{History and perspectives}\label{S 1.2}
As mentioned above, the fundamental groups of complements of complex curves are in general not easy to compute. 
In some cases, we know a priori that the group is Abelian, and having a presentation clearly suffices to compute the fundamental group very explicitly. 
This is the case, for instance, of \emph{nodal} curves, curves whose singularities are only double points: see Zariski~\cite[Section~3]{Zariski}, Fulton~\cite{fulton}, Deligne~\cite{Deligne}.

In~\cite{zar2}, Zariski computed the fundamental groups of the complements of the now-called \emph{Zariski sextics}: this computation shows that the fundamental group of the complement of a complex curve is not in general determined by the degree of the curve and the types of its singularities, but it also depends on the position of the singularities. 
This lead to the definition of \emph{Zariski pairs}, pairs of curves with the same combinatorics (same degree, same singularities, same incidence relations between components) but with non-homeomorphic complements. 
For an excellent historical and mathematical account of Zariski pairs, we recommend the survey by Artal Bartolo, Cogolludo Agust\'in, and Tokunuga~\cite{ACT-survey}.

Zariski's examples are two irreducible sextics with six singular points, all of which are simple cusps (locally modelled on $\{y^2 = x^3\} \subset \C^2$), and therefore of genus $4$. One of the two sextics, $\sC_1$, has all of its singular points lying on a conic, and has fundamental group $\Z/2\Z * \Z/3\Z$, while the other one, $\sC_2$, does not, and it has Abelian fundamental group. Libgober
~\cite[Section~7]{Libcyclic} observed that the fact that the six cusps are not in general position translates into a non-vanishing result for the Alexander polynomial (see~Remark~\ref{r:noLib} below for some considerations on what happens in symplectic setup). More precisely, he proved that, with respect to a generic complex line $L$, $\Delta_{\sC_1, L}(t) = t\inv - 1 + t$, while $\Delta_{\sC_2,L}(t) = 1$. That is to say, the Alexander polynomial is an effective isotopy invariant of curves.

Currently, we are not aware of any polynomial that appears as the Alexander polynomial of a symplectic curve, but not of a complex curve. 
It would be interesting to find a symplectic curve $\sC$ such that $\Delta_{\sC,L}$ is not the Alexander polynomial of any complex curve of the same degree and with the same singularities (assuming that there exists one\footnote{Note that there exists examples of symplectic curves whose singularity types are realised by no complex curve; see~\cite[Section~8]{golla_starkston_2022} for a detailed discussion.}).
It would also be interesting to find two symplectic curves $\sC \cup L$, $\sC' \cup L'$ of the same degree and with the same singularities, that are not isotopic to any complex curve and such that $\Delta_{\sC,L} \neq \Delta_{\sC',L'}$.

As mentioned above, Libgober has defined and studied Alexander polynomials of complex curves (and, more generally, the homotopy type of their complement) in a series of papers~\cite{Libcyclic, Lib1,L2}. He has proved the aforementioned divisibility relations and the relationship between the position of the singularities and the Alexander polynomial. For a curve $\sC \subseteq \mathbb{C}\mathbb{P}^2$, he also gave an explicit CW complex (of dimension $2$) with the same homotopy type as the complement of $\sC$~\cite{Lib1}. In~\cite{GS2}, the second author and Starkston gave explicit (Stein) handlebody decompositions of complements of certain rational cuspidal curves, namely those defined by the equations $\{x^pz = y^{p+1}\}$ and $\{(yz - x^2)^p = xy^{2p-1}\}$ and a (smooth) handlebody decomposition of the (unique) curve of degree $8$ with a singularity of topological type $\{x^3 = y^{22}\}$. (Handlebody decompositions for complements of other rational curves with one cusps were given by Lekili and Maydanskiy~\cite{LekiliMaydanskiy}, albeit not phrased in this language.)

In fact, part of this work relies on ideas developed in~\cite{Lib1} and in a recent preprint by Sugawara~\cite{sugawara2023handle}. Sugawara promotes Libgober's cell complex to a $2$-handlebody, and their result will be one of the two crucial steps towards proving Theorem~\ref{t:main}. That complements of complex curves admit such a decomposition is a well-known fact, as curve complements are affine varieties, and hence Stein surfaces. As far as we know, whether the complement of a symplectic curve is a Stein/Weinstein manifold is an open question.

We expect that the topological statements we prove here can be used to extend some generalisations of Libgober's work to the symplectic setting. For instance, in~\cite{CogolludoFlorens} Cogolludo Agust\'in and Florens proved that
\begin{equation}\label{CF1}
(\Delta^{\phi}_{\sC,L})^2\ \vert \prod_{i=1,\dots,s,\infty} \Delta^{\phi_{K_i}}_{K_i} \cdot \prod_{j=1}^{\ell}(1-t^{a_j})^{s^{\rm aff}_j-\chi(\sC^{\rm aff}_j)},\end{equation}
here $\sC^{\rm aff}_j$ is  a component of the affine curve $\sC^{\rm aff}=\sC \setminus L=\bigcup_{j=1}^{\ell} \sC^{\rm aff}_j$, $s^{\rm aff}_j$ and $\chi(\sC^{\rm aff}_j)$ are the number of singularities of $\sC^{\rm aff}_j$ and its Euler characteristic respectively.\\For the Alexander polynomial, \ref{CF1} gives the following
\begin{equation}\label{CF}
(\Delta_{\sC,L})^2 \mid \hat\Delta_{K_1}\cdots \hat\Delta_{K_\nu}\cdot \Delta_{K_\infty}\cdot (1-t)^{s^{\rm aff}-\chi(\sC^{\rm aff})}, 
\end{equation}
with $s^{\rm aff}$ the number of singularities of $\sC^{\rm aff}$ (see~\cite[Theorem 2.10]{ACT-survey} and Remark \ref{Rk CF} below for a consideration of this divisibility in the symplectic setting).
In fact, they work with twisted Alexander polynomials and with Reidemeister torsion, and they explicitly compute the ratio between the two quantities above.
\begin{rk}\label{Rk CF}
Let $\ell$ be the number of irreducible components of $\sC$ and $\Cns$ the non-singular part of $\sC$, and $\sC^{\rm aff, ns}$ the non-singular part of $\sC^{\rm aff}$. For notational convenience,  $\sC_0$ will denote as well the line at infinity $L$.\\
By combining the two assertions of Theorem~\ref{t:main} and direct computations, we obtain the following
\[
(\Delta^{\phi}_{\sC,L})^2 \mid \hat\Delta_{K_1}\cdots \hat\Delta_{K_\nu}\cdot \Delta_{K_\infty}\cdot \prod_{j=0}^{\ell} (1-t^{a_j})^{-\chi(\Cns_j)+2}
\]
and if all $a_j=1$, we have that \[
(\Delta_{\sC,L})^2 \mid \hat\Delta_{K_1}\cdots \hat\Delta_{K_\nu}\cdot \Delta_{K_\infty}\cdot (1-t)^{-\chi(\sC^{\rm aff, ns})}\cdot (1-t)^{\ell -s^{\rm aff}}
\]
One can notice that, since $s^{\rm aff}- \chi({\sC^{\rm aff}})=-\chi(\sC^{\rm aff, ns})$, what changes from \eqref{CF} is the factor $(1-t)^{\ell-{s^{\rm aff}}}$ which \textit{a priori} can not be controlled (that is, the sign of $\ell-s^{\rm aff}$ varies depending on the configuration of $\sC \cup L$).
\end{rk}

In a different direction, it would be interesting to study other invariants associated to the fundamental group, like characteristic varieties. For instance, Arapura proved that the first characteristic variety of a complex curve is a union of complex tori which comes with a fairly restrictive embedding in a larger complex torus~\cite{Arapura} (see also~\cite[Section~2.2]{ACT-survey}). We do not know whether similar results hold in the symplectic category.

\subsection*{Organisation of the paper}
In Section~\ref{S2}, we mention some generalities on Alexander and Oka polynomials of symplectic plane curves and give a description of a handle decomposition of curve complements. In Section~\ref{S3}, after fixing the context, we prove the Alexander/Oka polynomials divisibility for symplectic curves.  A new proof of the Alexander polynomials divisibility relation for algebraic curves is given in the Appendix~\ref{appendix}.

\subsection*{Notation and conventions}
We denote the iterated connected sum of $k$ copies of a manifold $M$ by $M^{\#k}$.
We denote the direct sum (respectively, the free product) of $k$ copies of a group $G$ with $G^{\oplus k}$ (resp. $G^{*k}$).
Unless otherwise specified, homology and cohomology will be taken with integer coefficients.
An equality of Alexander polynomials is always to be understood up to invertible elements in the ring (which usually is $\Q[t,t\inv]$, so the equality is up to factors of the form $qt^k$).

\subsection*{Acknowledgements}
We would like to thank Erwan Brugall\'e, Anthony Conway, 
  Laura Starkston, Enrique Artal Bartolo and Jos\'e Ignacio Cogolludo for several interesting discussions and for providing us with some useful references.
We would also like to thank Thomas Kragh for pointing us in the direction of Remark~\ref{r:noLib}, and Vincent Florens for suggesting working with Oka polynomials. We were both supported by the \'Etoile Montante PSyCo from the R\'egion Pays de la Loire. H.A. was also supported by the Spanish Ministry of Science through the Severo Ochoa Grant SEV
2023-2026 and through the research project PID2020-114750GB-C33 and by the Basque
Government through the BERC 2022-2025 program.


\section{Preliminaries on Alexander polynomials and curves}\label{S2}

In this section, we will cover some preliminaries on symplectic curves and their complements, and on Alexander polynomials of knots and of curves.

\subsection{Symplectic curves}\label{ss:sympl_curves}

We recall here the basic definitions of symplectic curves, following~\cite{golla_starkston_2022}.

A \emph{symplectic curve} $\sD$ is the image $u(\Sigma)$ of a $J$-holomorphic map $u \co (\Sigma,j) \to (\CP, J)$, for some (possibly disconnected) Riemann surface $(\Sigma,j)$ and some $\omega_{\rm FS}$-compatible almost-complex structure $J$. If the map is one-to-one away from finitely many points, 

we say that $u$ (or, with a slight abuse of terminology, $\Sigma$) is the \emph{normalisation} of $\sD$. The image of a connected component of the normalisation $u$ of $\sD$ is called an \emph{irreducible component} of $\sD$. We denote with $\Sing(\sD)$ the set of points where $u$ fails to be one-to-one or where $du = 0$.

The \emph{geometric genus} $g_i$ of an irreducible component $D_i$ of $\sD$ is the genus of the corresponding connected component of $\Sigma$ in the normalisation.

Recall also from~\cite{Mcduff} and~\cite{MicallefWhite} that symplectic curves have singularities that are homeomorphic to singularities of plane complex curves. (More on these singularities in the next subsection.)

A partial converse to McDuff's and Micallef and White's results ensures that a singular symplectic submanifold with locally holomorphic singularities can be made $J$-holomorphic for some $J$. This allows us to play with \emph{local} deformations in the symplectic setup, and gives us some flexibility in our constructions. As an example, let us consider Libgober's relations between the position of the singularities of a curve and its Alexander polynomial. (We are grateful to Thomas Kragh for an enlightening conversation that leads to the following remark.)

\begin{rk}\label{r:noLib}

No `if and only if' analogue of Libgober's statement holds in the symplectic setup, in the sense that the position of the cusps (with respect to other $J$-holomorphic curves) does not affect the Alexander polynomial.
To be concrete, consider the Zariski sextic $\sC_1$ defined in Section~\ref{S 1.2}, whose six cusps lie on a same conic. 
The fact that six cusps of $\sC_1$ lie on a same $J$-holomorphic conic is not preserved under perturbations of $J$, not even through almost-complex structures for which $\sC$ is $J$-holomorphic. 
Let $J_{\rm st}$ denote the  standard complex structure, so that $\sC_1$ is $J_{\rm st}$-holomorphic. 
By construction, there is a $J_{\rm st}$-holomorphic conic $Q$ passing through the six cusps of $\sC_1$. 
Choose also a generic $J_{\rm st}$-holomorphic line $L$. 
One can take a small and generic symplectic perturbation $Q'$ of $Q$ supported in the neighbourhood of a cuspidal point $p$ of $\sC_1$, so that $Q'$ no longer passes through $p$. 
This perturbation will make $\sC_1 \cup Q' \cup L$ be everywhere symplectic, with singularities that are modelled on singularities on complex curve: either because we have not deformed $Q$ or because we created some transverse double points (near $p$). 
Results of McDuff~\cite{Mcduff} and Micallef and White~\cite{MicallefWhite} guarantee that $\sC_1 \cup L \cup Q'$ is $J$-holomorphic for some almost-complex structure $J$. 
Since we have not perturbed either of $\sC_1$ or $L$, the Alexander polynomial of $\sC_1$ with respect to $L$ is $t\inv-1+t$, but there is no $J$-holomorphic conic passing through all six cusps of $\sC_1$. 
(There is a unique $J$-holomorphic conic passing through five points~\cite{Gromov}.)

It is possible that one of the directions of Libgober's theorem holds in the symplectic setup.
Very concretely: suppose that $\sC \cup Q \cup L$ is a symplectic curve such that $Q$ is a conic passing through five singularities of $\sC$ that are simple cusps, and $L$ is a generic line. 
Is $\Delta_{\sC \cup L}$ divisible by $t\inv - 1 + t$?
\end{rk}

Given a symplectic curve $\sD \subset \CP$ and $p \in \sD$, denote with $\beta(p)$ the number of branches of $\sD$ meeting at $p$. In particular $\beta(p) = 1$, if and only if $\sD$ is non-singular at $p$ or has a cuspidal singularity at $p$. For instance, a node $p$ has $\beta(p) = 2$.

\begin{lem}\label{b1}
Let $\sD$ be a symplectic curve in $\CP$ with $c$ irreducible components, $D_1,\dots D_c$. As a subspace of $\CP$,
\[
\chi(\sD) = \sum_{j = 1}^c (2-2g_j) - \sum_{p \in \Sing(\sD)} (\beta(p)-1).
\]
In particular, $b_1(\sD) = 2\sum_j g_j + \sum_{p \in \Sing(\sD)} (\beta(p)-1) - c + 1$.
\end{lem}

\begin{proof}
The normalisation map $\tilde\sD \to \sD$ is an isomorphism away from the preimage of the singular points. At each singular point $p \in \Sing(D)$, it collapses $\beta(p)$ points in $\tilde\sD$ onto a single point in $\sD$. It follows that $\chi(\sD) = \chi(\tilde\sD) - \sum_p (\beta(p)-1)$. The first statement readily follows by definition of the geometric genus.

For the computation of $b_1(\sD)$, it suffices to observe that $\sD$ is connected as a subset of $\CP$, since every two curves intersect at least once, and that $H_2(\sD)$ is freely generated by the fundamental classes of its components, so its rank is $c$.
\end{proof}

We now turn to studying neighbourhoods of symplectic curves (or, more generally, of PL-immersed surfaces).
The description of a closed regular neighbourhood $N_{\sD}$ of a symplectic curve $\mathcal{D}$ in a symplectic surface we give is inspired by~\cite{BorodzikHeddenLivingston, BCG1}, where the case of cuspidal curves is spelled out in detail.

Let us set up some notation first. As above, call $D_1, \dots, D_c$ the irreducible components of $\sD$. Call their geometric genera $g_1, \dots, g_c$ and their degrees $d_1, \dots, d_c$, respectively. Let $s$ be the number of singularities of $\sD$, and let $\{p_1, \dots, p_s\}$ be the singular points of $\sD$. Denote with $\beta_{ij}$ the number of branches of the $j\th$ component at the point $p_i$, so that $\beta_{i1}+\dots+\beta_{ic} = \beta(p_i)$ for each $i$. Denote also with $K_i$ 
 the link of the singularity of $\sD$ at $p_i$. For every $i$, each component of 
$K_i$ corresponds to a branch of the singularity of $\sD$ at $p_i$, which in turn belongs to the component $D_j$ for some $j$: we \emph{colour} this component of $K_i$ with $j$.

In order to construct the neighbourhood $N_\sD$ of $\sD$, first we take a small 4-ball $H^0_i$ for each singular point of $\sD$, and a collection of paths $\gamma_k \subset \sD \setminus \Int(H^0_i)$ with endpoints in $\cup_i\del H^0_i$ and cutting $\sD$ into a union of exactly $c$ discs. Note that we need exactly $b_1(\sD) + s - 1$ arcs to do so, and we know from Lemma~\ref{b1} that $b_1(\sD) = 2\sum_i g_i + \sum_{j} (\beta(p_j)-1) - c + 1$, so that the total number of arcs needed is $2\sum_i g_i + \sum_j \beta(p_j) - c$. Call $H^1_k$ a regular neighbourhood of $\gamma_k$, intersecting each $H^0_i$ only in $\del H^0_i \cap \del H^1_k$ and $\sD$ in a regular neighbourhood (in $\sD$) of $\gamma_k$.

A handle decomposition of a regular neighbourhood $N_\sD$ of $\sD$ is now given by:
\begin{itemize}
\item one 0-handle $H^0_i$ for each $p_i$;
\item one 1-handle $H^1_k$ for each $\gamma_k$, obtained by thickening $\gamma_k$---observe that 
the boundary of the 1-handlebody $N^1_\sD$ constructed so far intersects $\sD$ in a collection of exactly $c$ knots, each corresponding to a component of $\sD$, or, equivalently, to a colour between $1$ and $c$;
\item one 2-handle $H^2_j$ for each component of $\sD$, attached along the $j$-coloured component of the intersection $\sD\cap \del N^1_\sD$, where $N^1_{\sD}$ is the 1-skeleton we constructed so far, each with framing $d_j^2$.
\end{itemize}

Note that $N^1_\sD$ is connected, since $\sD$ is, and that $\del N^1_\sD = (S^1 \times S^2)^{\# b_1(\sD)}$. For notational convenience, from now on let $b := b_1(\sD)= 2\sum_i g_i + \sum_{j} (\beta(p_j)-1) - c + 1$. We want to describe the attaching link $J = J_1 \cup \dots \cup J_c \subset (S^1 \times S^2)^{\# b}$ of the 2-handles $H^2_1, \dots, H^2_c$.

\begin{defi}
The \emph{Borromean knot} $B_{g}\subset (S^1\times S^2)^{\# 2g}$ is described in Figure~\ref{f:borromean-knot}.
\end{defi}

For instance, $B_1$ is the knot obtained by doing 0-surgery along two components of the Borromean rings (Figure~\ref{f:borromean-rings}), and $B_{g+1} = B_g \# B_1$.

\begin{figure}
\labellist
\small\hair 2pt
\pinlabel $\underbrace{\hphantom{----------------------.}}_g$ at 228 45
\pinlabel $B_{2g}$ at -20 20
\endlabellist
\centering
\includegraphics[width=0.75\textwidth]{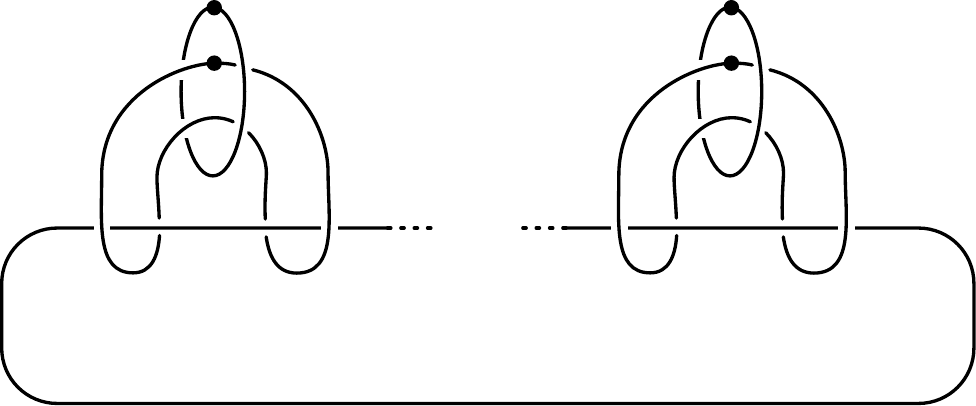}
\caption{The Borromean knot $B_{2g}$. The dotted curves represent 1-handles (or 0-surgeries), see~\cite{GompfStipsicz}.}\label{f:borromean-knot}
\end{figure}

\begin{figure}
\includegraphics[width=0.3\textwidth]{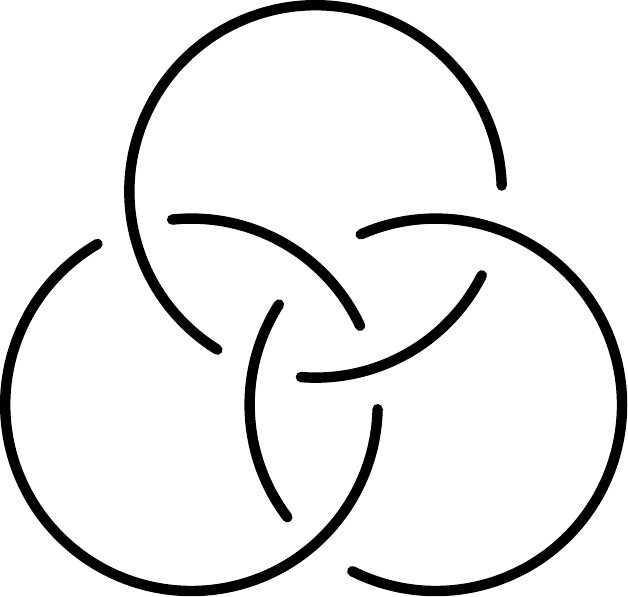}
\caption{The Borromean rings.}\label{f:borromean-rings}
\end{figure}

\begin{defi}
Let $K$ be a link in a 3-manifold $M$, coloured by integers $i_1, \dots, i_{c'}$ and such that each colour appears at least once. Choose two components $K_1, K_2$ of $K$ with the same colour, and knotify them, to get a coloured link 
$K' \subset M \# S^1\times S^2$ with one less component than 
$K$. Iterate this procedure until you get a link 
$\kappa_\col(K) \subset M_\col$ with $c'$ components, each coloured by an index $i_1, \dots, i_{c'}$. We call $(M_\col, \kappa_\col(K))$ 
the \emph{coloured knotification} of $K$
\end{defi}

The following proposition formalises the construction of the attaching link $J$ of the $2$-handles of $N_\sD$. The proof is omitted, as it is an easy adaptation of~\cite[Theorem~3.1]{BorodzikHeddenLivingston} or~\cite[Lemma~3.1]{BCG1}.

\begin{prop}
The attaching link $J$ of the $c$ $2$-handles of $N_\sD$ is the \emph{coloured knotification} of the \emph{coloured connected sum} of all links of singularities of $\sD$ followed by a coloured connected sum with the $j$-coloured Borromean knot $B_{g_j}$ for $j=1,\dots,c$. The framing of the $j$-coloured component of $J$ (which is null-homologous in $\del N^1_\sD$) is $d_j^2$.
\end{prop}

See Figure~\ref{f:example-handledec} for an example.

\begin{figure}
\includegraphics[width=0.65\textwidth]{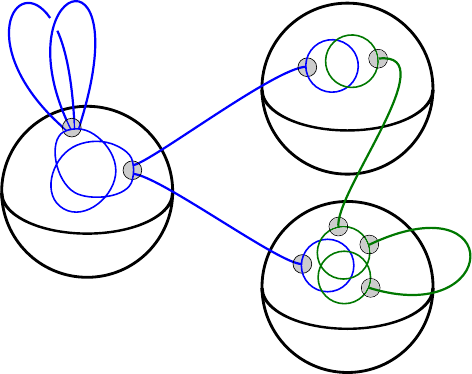}
\caption{A schematic example of a handle decomposition of the neighbourhood of a reducible curve: here the curve has two components $C_{\rm b}$ and $C_{\rm g}$, and we colour the components of the links of their singularities by blue or green according to which curve they belong to. In the left-most ball, we see a trefoil and the two blue arcs represent a connected sum with a (blue) Borromean knot $B_1$, while the other blue components are unknots, so $C_{\rm b}$ has genus 1 and one simple cusp. In the bottom-most ball we see a green Hopf link, and the arc connecting its two components represents their knotification, so $C_{\rm g}$ has genus 0 and one singularity which is a simple node. $C_{\rm b}$ and $C_{\rm g}$ intersect at two points, which can be seen from the two balls on the right. On the top one we see a Hopf link, so $C_{\rm b}$ and $C_{\rm g}$ are non-singular and intersect transversely at that point. On the bottom ball we see a triple Hopf link (three fibres of the Hopf fibration), two green and one blue: this means that $C_{\rm b}$ and $C_{\rm g}$ intersect at the node of $C_{\rm g}$, in a generic way. The other arc connect the various coloured links along the corresponding curves. In this case, the result is a 2-component link in $(S^1\times S^2)^{\#4}$.}\label{f:example-handledec}
\end{figure}


\subsection{Handle decompositions of complex curve complements}\label{Handle}
Recall that the complement of a complex curve is an affine complex surface, hence it is a Stein surface, and as such it admits a handle decomposition with handles of indices 0, 1, and 2. In a similar spirit, Libgober had given an explicit 2-dimensional finite CW complex with the same homotopy type as the complement~\cite{Lib1}. Very recently, Sugawara promoted Libgober's cell complex to a handle decomposition of the complement. In this section, we briefly recall Sugawara's work, which we will adapt to the symplectic context to prove the divisibility relations. 

To fix the framework, we recall first the construction used to define the braid monodromy of a plane algebraic curve which allows to understand the topology of a (singular) plane curve, to find a representation of the fundamental group of its complement and to give an explicit handle decomposition of the latter (see~\cite{Mo, Lib1, Cohen} for more details).

Let $\sC$ be an algebraic curve in $\C\P^2$ of degree $d$ with $\ell$ components. 
Let $L$ be a line at infinity with $\sC^+=\sC \cup L$. After a suitable change of coordinates, we may assume that $p=[0:0:1]$ 
belongs to $L$ but not to $\sC$.

Consider the linear projection from $p$, $\pi \co \C\P^2 \setminus L \to \C$.
Let $\mathcal{Y}=\{y_1,\dots,y_s\}$ be the set of points in $\C$ for which the fibers of $\pi$ contains at least one singular point of $\sC$ or are tangent to $\sC$. 
If each point $\pi\inv(y_i)$ either contains a unique singularity of $\sC$ or is simply-tangent to $\sC$ at a single point for each $i$, we say that the projection $\pi$ is \emph{generic}.
For each $i$, let $\gamma_i$ be a loop based at $y_0 \in \C\setminus \mathcal{Y}$ which goes counterclockwise around $y_i$ once does not wind around any other $y_j$ for $j\neq i$.

The fibers of $\pi$ are complex lines $L_y := \pi\inv(y) \simeq \C$. If $y \not \in \mathcal{Y}$, then $L_y$ intersects $\sC$ transversely in $d$ points. The map $\pi$ is a singular fibration $\C\P^2\setminus\sC^+ \to \C$, and its restriction to the complement of $\pi\inv(\mathcal{Y})$ is a locally trivial fibration
\[
\CP \setminus(\sC^+ \cup L_{y_1} \cup \dots \cup L_{y_s}) \to \C \setminus \mathcal{Y}
\]
with fiber $F \cong \C \setminus \{1,\dots,d\}$. Since the mapping class group of $F$ (relative to the boundary at infinity) is the braid group $B_d$ on $d$ strands, the monodromy of this fibration is a morphism
\[
\pi_1(\C \setminus \mathcal{Y}, y_0) \to B_d,
\]
which is called the \emph{braid monodromy} of the curve $\sC$ with respect to the line at infinity $L$. Concretely, this morphism takes the loop $\gamma_i$ to a braid describing the motion of the $d$ points in the intersection of the curve and the fibers as one moves along $\gamma_i$.

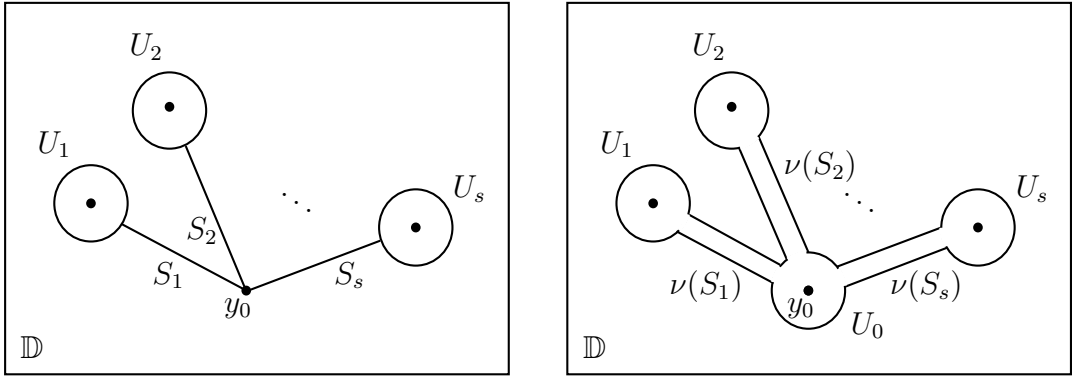
\begin{figure}[ht]
\labellist
\pinlabel $\hD$ at 10 10
\pinlabel $y_0$ at 84 24
\pinlabel $U_1$ at 18 83
\pinlabel $U_2$ at 51 118
\pinlabel $U_s$ at 167 67
\pinlabel $S_1$ at 59 36
\pinlabel $S_2$ at 71 52
\pinlabel $S_s$ at 124 36
\pinlabel $\ddots$ at 105 65
\pinlabel $\hD$ at 212 10
\pinlabel $y_0$ at 286 24
\pinlabel $U_0$ at 310 18
\pinlabel $U_1$ at 220 83
\pinlabel $U_2$ at 253 118
\pinlabel $U_s$ at 369 67
\pinlabel $\nu(S_1)$ at 252 32
\pinlabel $\nu(S_2)$ at 293 75
\pinlabel $\nu(S_s)$ at 331 32
\pinlabel $\ddots$ at 307 65
\endlabellist
\includegraphics[scale=1.05]{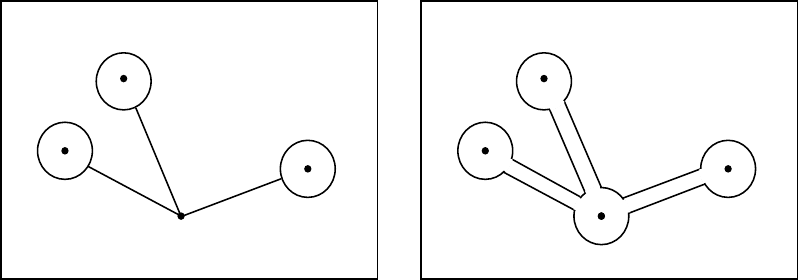}
\caption{\label{1l} $\Gamma'$ and its neighbourhood $\nu(\Gamma')$ in $\mathbb{D}.$}
\end{figure}

We denote by $U_i$ a small disc in $\C$ centered at $y_i \in \mathcal{Y}$ for $i\in\{1,\dots,s\}$, and by $S_i$ a system of non-intersecting paths connecting the basepoint $y_0$ of $\C$ with a point $y_i'$ on the boundary of $U_i$. We denote by $\hD$ a sufficiently large disk containing the system $\bigcup_i (S_i \cup U_i)$.
Let $S_0$ be a path connecting $y_0$ to $\del \hD$, intersecting $\bigcup_i (S_i \cup U_i)$ only in $y_0$.
Trivialise $\pi$ over $\hD$ and let $\hD_y$ be a disc in the fibres (with respect to this trivialisation) that contains all intersections of the fibres of $\pi$ and $\sC$.
$\hD\times \hD_y$ contains a regular neighbourhood $\nu(\sC)$ of $\sC$, and this assumption on $\hD_y$ guarantees that $\nu(\sC) \cap \del(\hD\times\hD_y)$ is contained in $(\del\hD) \times \hD_y$ (in other words, $\sC$ does not escape from $\hD\times \hD_y$ ``vertically'', but only ``horizontally''); see Figure~\ref{2l}.

Note that the interior of $(\hD\times\hD_y)\setminus \nu(\sC)$ is diffeomorphic to $\C\P^2\setminus \sC^+$. We denote by $\pi_0$ be the restriction of $\pi$ to $(\hD\times\hD_y)\setminus \nu(\sC)$.

With the above setting, Sugawara gives an explicit handle decomposition of the complement of a plane curve. We recall first an easy Lemma.

\begin{lem}[\cite{sugawara2023handle}]\label{LemSug}
Let $X$ and $Y$ be compact manifolds with boundary, with $\dim X = \dim Y + 1$. Let $Y^+$ denote $[0,1]\times Y$ and $Y_t = \{t\} \times Y \subset Y^+$.
\begin{itemize}
\item If $\iota \colon Y^+ \hookrightarrow X$ is an embedding with $\iota(Y^+) \cap \del X = \iota(Y_1)$, then, after smoothing corners, $\overline{X \setminus \iota(Y^+)}$ is diffeomorphic to $X$.
\item If $\iota_0\colon Y_0 \to \del X$ is an embedding, then, after smoothing corners, $X \cup_{\iota_0} Y^+$ is diffeomorphic to $X$.
\end{itemize}
\end{lem}

\begin{figure}
\labellist
\pinlabel $\nu(\sC)$ at 120 145
\pinlabel $\hD$ at 30 10
\pinlabel $\hD_y$ at 260 120
\pinlabel $U_0$ at 60 33
\pinlabel $\nu(S_i)$ at 110 33
\pinlabel $B_i$ at 138 33
\pinlabel $V_i$ at 167 33
\pinlabel $A_i$ at 182 45
\pinlabel $U_i$ at 199 55
\endlabellist
\includegraphics[scale=1.2]{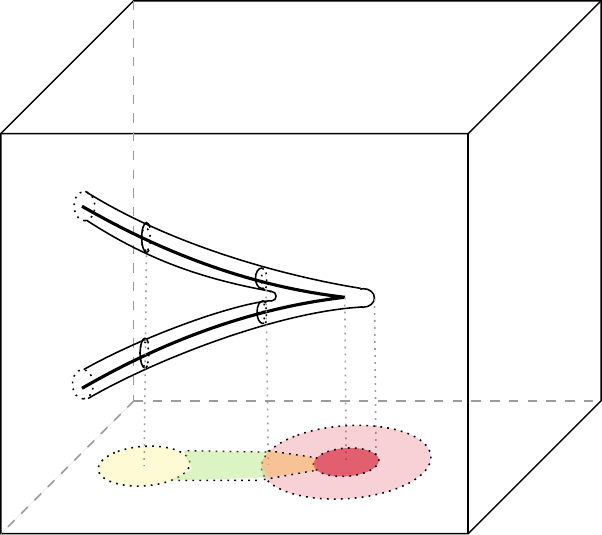}
\caption{The decomposition of $(\hD\times \hD_y)\setminus \nu(\sC)$.}\label{2l}
\end{figure}

The main idea for a handle decomposition of $\C\P^2 \setminus \sC^+$ is to decompose it into multiple parts starting mainly with the neighbourhood of each singularity and study their diffeomorphism types then see how each part is attached to another as follows. Let $\nu(\Gamma)= \bigcup_{i=1}^{s} U_i \cup \bigcup_{i=0}^{s} \nu(S_i)$, the union of small discs around the singular values of $\pi$ in $\hD$ and a neighbourhood of the  non-intersecting arc system connecting the discs to the base point and the one connecting the base point to $\del \hD$ (see \cite[Figure 4]{sugawara2023handle}).
In this setting, $\hD \setminus \nu(\Gamma)$ is contractible and $\pi_0^{-1}(\hD \setminus \nu(\Gamma))$ is diffeomorphic to a 4-manifold with $d$ $1$-handles attached. It remains to attach $\pi_0^{-1}(\nu(\Gamma))$ to $\pi_0^{-1}(\hD \setminus \nu(\Gamma))$. For this, $\nu(\Gamma)$ is decomposed into multiple parts (see \cite[Figure 7]{sugawara2023handle}). Mainly, let $V_i\subset U_i$ be a small disk around $y_i$ such that $\pi_0^{-1}(v) \simeq \mathbb{D}\setminus \{d-(m_i-1)\text{ points}\}$, for all $v \in V_i$ and $m_i$ the multiplicity of the singularity of $\sC$ contained in $\mathcal{L}_{y_i}$. The complement of $V_i$ in $U_i$ is decomposed into two parts: $A_i$ and $B_i=\overline{U_i\setminus(A_i \cup V_i)}$.
 
Attaching $\pi_0^{-1}(A_i \cup V_i)$ to $\pi_0^{-1}(\hD \setminus \nu(\Gamma))$ does not change the diffeomorphism type by Lemma \ref{LemSug}. However, attaching $\pi_0^{-1}(B_i)$  to what precedes is equivalent to attach $(m_i-1)$ 2-handles. To prove this, one has to consider stable map on the boundary of $(\hD\times \hD_y)\setminus \nu(\sC)$ and use local coordinates description. Finally, attaching the rest $\pi_0^{-1}(\nu(\Gamma)\setminus \cup_i U_i)$ does not change the diffeomorphism type, again by Lemma~\ref{LemSug}. Hence, $\C\P^2 \setminus \sC^+$ has a handle decomposition with handles of indices at most 2.


\subsection{Alexander polynomials} \label{AP}
We recall first the definition of classical Alexander polynomials, of links of singularities, and of links at infinity.

If $K$ is an oriented $\ell$-component link in $S^3$, then by Alexander duality $H_1(S^3\setminus K) \cong \Z^\ell$, with a preferred basis given by the meridians of the components of $K$. 
Using this basis, we construct a morphism $\overline\phi_K\co H_1(S^3\setminus K) \to \Z$, sending each meridian to $1$. 
By pre-composing with the Abelianisation map\footnote{We omit basepoints here. This will neither play a role nor create any confusion.} $\pi_1(S^3\setminus K) \to H_1(S^3\setminus K)$, we get a homomorphism $\phi_K \co \pi_1(S^3\setminus K) \to \Z$, and a corresponding infinite cyclic cover $Y \to S^3\setminus K$. 
The group of deck transformations of $Y \to S^3 \setminus K$ is $\Z$, which endows $H_*(Y;\Z)$ with the structure of a $\Z[t,t\inv]$-module. 
As a $\Z[t,t\inv]$-module, $H_1(Y;\Z)$ decomposes as a sum of cyclic modules $\bigoplus_{i}\frac{\Z[t,t^{-1}]}{(\lambda_i(t))}$ (see, for instance,~\cite[Corollary~8.C.5]{Rolfsen}).
The Alexander polynomial of $K$, denoted with $\Delta_K$, is then $\Delta_K(t) := \prod_i \lambda_i(t)$, and it is a Laurent polynomial, defined up to units in $\Z[t,t\inv]$, i.e. up to a factor fo the form $\pm t^k$ for some integer $k$.

The two crucial points of the definition above are that $K$ is of codimension $2$ in $S^3$, and that there is a preferred morphism from $H_1(S^3\setminus K) \to \Z$ (associated to an orientation of $K$). At least considering the first point alone, it makes sense to try to generalise the definition to (symplectic) curves in $\CP$, which are codimension-$2$ objects.

If $\sC \subset \CP$ is a curve with $\ell$ components whose degrees are $d_1, \dots, d_\ell$, however, $H_1(\CP\setminus \sC) = \Z^\ell/\langle(d_1, \dots, d_\ell)\rangle$. There is no preferred morphism to $\Z$: for instance, there is no homomorphism sending all meridians to $1$. Libgober's idea was to introduce an auxiliary line $L$ to take care of this issue.

Let now $L$ be an auxiliary line and denote $\sC^+ = \sC \cup L$. We now have that $H_1(\CP \setminus \sC^+) = \Z^{\ell+1}/\langle(1,d_1, \dots, d_\ell)\rangle$, and, since $L$ is a \emph{distinguished} component of $\sC^+$, there is a preferred morphism $\phi_{\sC,L}\co \pi_1(\CP\setminus \sC^+) \to \Z$: we send the meridian of each component of $\sC$ to $1$, and the meridian of $L$ to $-(d_1+\dots+d_\ell)$. As above, we have an associated infinite cyclic cover $X \to \CP\setminus \sC^+$, and $H_1(X;\Q)$ is endowed with a structure of $\Q[t,t\inv]$-module.

\begin{rk}

Equivalently, one can think of $X$ as being the infinite cyclic cover of $(\CP \setminus L) \setminus \sC$, where we view $\sC$ as an affine curve in $\CP\setminus L = \C^2$. In this case, Alexander duality tells us that $H_1(\C^2\setminus \sC)$ is indeed $\Z^\ell$, generated by the meridians of $\sC$, so we are back to the setup of classical links and we can (almost) run the same argument.

We prefer the first setup because compact symplectic curves in $\CP$ are a lot tamer than symplectic curves in $\C^2$---the latter being closer to analytic submanifolds of $\C^2$ than to algebraic curves in $\C^2$.
\end{rk}

Either way, we call the order of $H_1(X;\Q)$ the \emph{Alexander polynomial} of $\sC$ with respect to $L$, $\Delta_{\sC,L}$. As for $\Delta_K$ this is a Laurent polynomial, which is now defined up to units in $\Q[t,t\inv]$, that is, up to a factor of the form $qt^k$ for some rational number $q$ and integer $k$ (see \cite[Section~2]{ACT-survey}). Recall that the order of a $\Q[t,t\inv]$-module $M$ is defined as follows:

\begin{equation}\label{eq:order}
\ord_{\Q[t,t\inv]}M = \left\{
    \begin{array}{ll}
        0 & \mbox{if } M \mbox{ has a free summand} \\
        1 & \mbox{if } M \mbox{ is 0}\\
\prod_i \lambda_i(t) & \mbox{if }  M \cong \bigoplus_{i}\frac{\Q[t,t^{-1}]}{(\lambda_i(t))}.
    \end{array}
\right.
\end{equation}

\begin{rk}
Here we chose to use rational coefficients, but Alexander polynomials can be defined over any \emph{field}. Contrarily to the case of classical knots, working over $\Z$ is more delicate, and in fact not always possible. Indeed, there are examples of complex curves for which the $H_1(X;\Z)$ does not decompose as sum of cyclic modules, so the definition above cannot extend over $\Z$ in full generality. For instance, see~\cite[Example 2.5(2)]{ACT-survey}, which shows that when $\sC$ is the Steiner quartic (the unique quartic with three cusps, which is the dual to a nodal cubic) the module $H_1(X;\Z)$ is not cyclic.
\end{rk}

The definition can easily be extended to the case when $\sC$ is a symplectic curve: choose an almost-complex structure $J$ that is $\omega_\FS$-compatible and for which $\sC$ is $J$-holomorphic, and add a $J$-holomorphic line $L$.
Then run the same argument.
In fact, the definition can also be extended to collections of singular submanifolds of $\CP$ (e.g. PL-immersed surfaces in $\CP$), but in this case the choice of the auxiliary ``line'' is a lot more arbitrary.

\subsection{Oka polynomials}

With the Alexander polynomial for complex curves in mind, we now go back to the classical case, that of links in $S^3$, and we define the variant $\hat\Delta$ of the Alexander polynomial appearing in Theorem~\ref{t:main}. In fact, we will define Oka polynomials, which are a suitably twisted version of Alexander polynomials, of which $\hat\Delta$ is a special case.

As already mentioned, the classical Alexander polynomial of a link $K$ comes from the choice of a preferred morphism $\phi_K \co \pi_1(S^3 \setminus K) \to \Z$ which sends all meridians of $K$ to $1$.
When $K$ is the link of a singularity of $\sC^+ = \sC \cup L$ lying on $L$, however, $K$ has a unique (unknotted) component that corresponds to the component $L$ of $\sC^+$.
As we have seen above, this component behaves differently from the others, in that its meridian is not sent to $1 \in \Z$ by the morphism $\pi_1(\CP\setminus \sC^+) \to \Z$, but rather to $-d$ (where $d$ is the degree of $\sC$).

More generally, we can consider an arbitrary morphism $\phi\co \pi_1(X) \to \Z$, where $X$ is either $S^3\setminus K$ or $\CP\setminus \sC^+$. We will only consider morphisms that sends each meridian (of $K$ or of $\sC^+$) to a non-zero integer. For a link $K$ with a marked component and an integer $d$ (which is left implicit in the notation) we denote with $\hat\phi_K \co \pi_1(S^3 \setminus K) \to \Z$ the morphism that sends the marked meridian of $K$ to $-d$ and every other meridian to $1$.

We consider Alexander polynomials coming from the infinite cyclic covers associated to these homorphisms $\phi$.

\begin{defi}
Let $K$ be a link in $S^3$ and $\phi\co \pi_1(S^3 \setminus K) \to \Z$ be a homomorphism. We define the \emph{Oka polynomial} $\Delta^\phi_K \in \Z[t,t\inv]$ associated to $(K,\phi)$ as the order of torsion of $H_1(X_\phi)$ as a $\Z[t,t\inv]$-module, where $X_\phi \to S^3\setminus K$ is the infinite cyclic cover of $S^3\setminus K$ associated to $\phi$.

Analogously, let $\sD$ be a reducible symplectic curve in $\CP$ and $\phi\co \pi_1(\CP \setminus \sD) \to \Z$ be a homomorphism. We define the \emph{Oka polynomial} $\Delta^\phi_\sD \in \Q[t,t\inv]$ associated to $(\sD, \phi)$ as the order of torsion of $H_1(X_\phi)$ as a $\Q[t,t\inv]$-module, where $X_\phi \to \CP\setminus \sD$ is the infinite cyclic cover of $\CP\setminus \sD$ associated to $\phi$.
\end{defi} 

\begin{rk}
We do not assume that $\phi$ be surjective. We do, however, keep assuming that $\phi$ does not send any meridian to $0$. If the image of $\phi$ is generated by $a > 0$, then $X_\phi$ comprises $a$ connected components.
\end{rk}

As a special case, we have the following definition.

\begin{defi}
Let $K$ be a link in $S^3$ with a marked component, and $d$ be an integer. We define
\[
\hat\Delta_K  := \Delta^{\hat\phi_K}_K\in \Z[t,t\inv].
\]
(In order to keep the notation more readable, we will leave the integer $d$ and the marked component of the knot implicit.)
\end{defi}

We will now focus on the case of knots.
Just like $\Delta_K$, $\Delta^\phi_K$ (and hence $\hat\Delta_K$, too), can be computed from the \emph{multi-variable Alexander polynomial} of $K$ (see~\cite[Chapter~7]{Kawauchi}). 
We use the same notation for the multi-variable Alexander polynomial, except that now, for a link $K$ comprising $n+1$ components, $\Delta_K$ is an object of $\Z[t_0^{\pm1}, t_1^{\pm1}, \dots, t_n^{\pm1}]$. If $K$ has a marked component, it will always be the one corresponding to the variable $t_0$.

\begin{prop}
Fix a non-negative integer $n$. 
Let $K$ be a link with $n+1$ components and $\phi$ be a homomorphism as above, with $\im \phi = a\Z$ with $a>0$.
\begin{enumerate}
\item If $n+1 = 1$, $\Delta^\phi_K(t) = \Delta_K(t^a)$.
\item If $n+1 > 1$, $\Delta^\phi_K(t) = (1-t^a)\Delta_K(t^{\phi(\mu_0)}, \dots, t^{\phi(\mu_n)})$.
\end{enumerate}
As a special case, if $K$ has its 0$^{\text{th}}$ component marked,
\[
\hat\Delta_K(t) = (1-t)\Delta_K(t^{-d},t,\dots,t).
\]
\end{prop}

\begin{proof}
We denote by $\phi^1$ (resp. $\phi^a$) the morphism sending the meridian to 1 (resp. to a). Let us consider the following commutative diagram:
$$\begin{tikzcd}
 \pi_1(X)   \arrow[r, "id"] \arrow[d,"\phi^1"]
 & \pi_1(X)  \arrow[d, "\phi^a"] \\
 \mathbb{Z} \arrow[r, "\times a"]
 & \mathbb{Z}
 \end{tikzcd} $$
We denote by $M_1$ (resp. $M_2$) the Fox matrix associated to $\phi^1$ (resp. $\phi^a$) with entries in $\mathbb{Z}[u^{\pm 1}]$ (resp. $\mathbb{Z}[t^{\pm 1}]$).  $(1)$ is then straightforward, noticing that $M_2$ is the same as $M_1$ replacing $u$ by $t^a$.\\
For (2), this follows directly from~\cite[Proposition~7.3.10]{Kawauchi}. The second half of the statement follows from the observation that $\im\hat\phi_K = \Z$, so $a = 1$.
\end{proof}

We now give three sample computations of $\hat\Delta$ in some simple cases.

\begin{example}\label{ex:deltahat}
Fix a positive integer $d>1$ throughout. 
Now, recall from~\cite[Theorem~12.1]{EisenbudNeumann}, applied to the splice diagram of the link $T(n+1,n+1)$ (the union of $d$ fibres of the Hopf fibration, oriented so that every two components have linking number $+1$) that can be found in~\cite[Page~448]{Neumann-linksatinfty}, that
\[
\Delta_{T(n+1,n+1)}(t_0, \dots, t_n) = (t_0\cdots t_n - 1)^{n-1}.
\]
\begin{itemize}
\item Let us consider the case $n = 1$. 
This is the Hopf link, $T(2,2)$, with one marked component, which corresponds to a transverse intersection of $L$ and $\sC$ (at a non-singular point of $\sC$).
Then
\[
\hat\Delta_{T(2,2)} = (1-t)\Delta(t^{-d},t) = 1-t.
\]

\item If $1 < n < d$, we have:
\[
\hat\Delta_{T(n+1,n+1)} = (1-t)\Delta(t^{-d},t,\dots,t) = (1-t)(t^{n-d}-1)^{n-1},
\]
This corresponds to the case where the curve $\sC$ has an ordinary $n$-fold point at infinity (i.e. $n$ non-singular branches meeting transversely), and each branch of $\sC$ is transverse to the line at infinity. This happens, for instance, for line arrangements with points on the line at infinity.

\item Let us consider the case $n=d$. A topological realisation of this case is discussed in Remark~\ref{r:vanishingAlex}.
Then
\[
\hat\Delta_{T(d+1,d+1)} = (1-t)\Delta(t^{-d},t,\dots,t) = 0.
\]

\item Finally, let us consider the case of the Alexander polynomial at infinity with respect to a generic line: in this case, the link at infinity is $T(d,d)$ and we have:
\[
\Delta_{T(d,d)}(t) = (1-t)(t^d-1)^{d-2}.
\]
\end{itemize}
\end{example}

\begin{rk}
For complex or symplectic curves, there is a ``preferred'' choice of line: namely, a generic line (say, one that intersects $\sC$ transversely) will always give the same Alexander polynomial.

In the case of complex curves, this is because being tangent to $\sC$ is a (complex) codimension-$1$ phenomenon in the dual projective plane $(\CP)^\vee$, and in particular the set of transverse lines is connected, so that the homotopy type of the complement does not depend on the choice.

In the case of symplectic curves, this is a consequence of, for instance,~\cite[Proposition~5.1]{golla_starkston_2022}. Indeed, that statement guarantees that a symplectic $L$ making $\sC \cup L$ symplectic, and meeting $\sC$ sufficiently generically, always exists and is unique up to isotopy. (So that, in fact, the isotopy class of $\sC \cup L$ is also uniquely determined.)
\end{rk}

As mentioned above, if $\sC^+$ is a symplectic curve, then its singular points are isolated, and they are \emph{topologically equivalent} to singularities of complex curves~\cite{Mcduff, MicallefWhite}. That is to say, for each singularity $p$ of $\sC^+$ there exist a complex curve $D \subset B^4 \subset \C^2$, an open neighbourhood $U$ of $p$, and a homeomorphism of triples $(U,\sC^+,p) \simeq (B^4,D,0)$. By work of Milnor~\cite{Milnor-hypersurfaces}, $(U,\sC^+,p)$ is homeomorphic to the cone over a link in $S^3$. This link, which is just the intersection $\sC^+ \cap \del U$, is called the \emph{link} of the singularity of $\sC^+$ at $p$. The links that arise in this fashion are often called \emph{algebraic}, and they form a well-understood class of links. (For the initiated reader, they are ``sufficiently positive'' cables of the unknot~\cite{Brauner, Kaehler}.)

It is well-known that Alexander polynomials of algebraic links are products of cyclotomic polynomials (see, for instance, \cite[Chapter~11]{Lickorish} or \cite[Exercise~7.D.8]{Rolfsen}). 
For the polynomials $\hat\Delta_K$ of marked algebraic links defined above, by~\cite[Theorems~9.4 and~12.1]{EisenbudNeumann} they are either $0$ or a product of cyclotomic polynomials.
Thus, Corollary~\ref{c:cyclotomic} is an easy consequence of Theorems~\ref{t:main1} and ~\ref{t:main}. Remark~\ref{r:cyclautomatic} follows from the computations in Example~\ref{ex:deltahat} and from (the proof of) Proposition~\ref{l-1}.

Finally, given a curve $\sC$ and a line $L$, we can define $K_{\infty}$ the \emph{link at infinity} of $\sC$ with respect to $L$ as the intersection $\sC \cap \del \nu(L) \subset \del \nu(L)$, where $\nu(L)$ is a sufficiently small tubular neighbourhood of $L$. The isotopy type of $\sC \cap \del \nu(L)$ will not depend on $\nu(L)$, as long as $\nu(L)$ is sufficiently small. Alternatively, one can think of the link at infinity in the following way: remove $L$ from $\CP$, thus obtaining $\C^2$ with an affine curve $\sC\setminus L$, and consider the intersection of the boundary of a very large ball in $\C^2$ with $\sC$.
\begin{rk}
As we have seen in Example~\ref{ex:deltahat}, there are some cases in which $\hat\Delta_{K_i}$ does not vanish. Is it possible to characterise precisely for which (marked) links of singularities $K$ the twisted Alexander polynomial $\hat\Delta_K$ vanishes? Is there a relationship between the vanishing of one of the $\hat\Delta_{K_i}$ and the vanishing of $\Delta_{K_\infty}$?
\end{rk}

\subsection{Algebraic and computational aspects of Alexander polynomials}

In principle, the Alexander polynomial (of a knot $K$ or of a curve $\sC$) can be algorithmically computed via Fox calculus starting from a presentation of $\pi_1(S^3\setminus K)$ or $\pi_1(\CP\setminus \sC^+)$. Let us recall briefly the idea.

Suppose that $\phi: \pi_1(Z) \to \mathbb{Z}$ is a surjective morphism (here $Z$ is $S^3\setminus K$ or $\CP\setminus \sC^+$). Assume that $\pi_1(Z)$ has the following presentation
 \[
 \pi_1(Z)=\langle x_1,\dots,x_n \mid   r_1,\dots,r_m \rangle.
 \]

 Thus we have a surjective morphism $\psi: F_n \to \pi_1(Z)$, where $F_n$ is a free group of rank $n$ generated by  $x_1,\dots,x_n$. Consider the group ring of $F_n$ with $\Q$-coefficient $\Q[F_n]$. The Fox differential $\frac{\del}{\del x_i}:\Q[F_n] \to \Q[F_n]$ is the unique $\Q$-linear morphism satisfying the following properties:
\begin{align*}
 \frac{\del x_i}{\del x_j} &= \delta_{ij},\\
 \frac{\del uv}{\del x_j} &= \frac{\del u}{\del x_j}+u\frac{\del v}{\del x_j}, \text{ for } u,v \in \Q[F_n].
\end{align*}
 The composition $\phi \circ \psi: F_n \to \mathbb{Z}$ gives a ring homomorphism $\theta: \Q[F_n] \to \Q[t^{\pm1}]$ and the associated \emph{Fox matrix} is an $n\times m$ matrix whose $(i,j)$-entry is given by
\[
\theta\left(\frac{\del r_i}{\del x_j}\right) \in \Q[t^{\pm1}].
\]
The Alexander polynomial of $Z$ relative to $\phi$ is the greatest common divisor of the minors of order $n-1$ of the Fox matrix if $n \leq m$ and $0$ otherwise.

Similar ideas can be used to compute the multi-variable Alexander polynomial.
See, for instance,~\cite[Theorem~7.1.5]{Kawauchi}.

More generally, using Fox calculus as above, one can define the Alexander polynomial $\Delta^{\phi}_{G}$ associated to any group $G$ with a surjective morphism $\phi\co G \to \Z$. With this notation, for instance, we have $\Delta_K = \Delta_{\pi_1(S^3\setminus K)}^ {\phi}$, where $\phi$ sends each meridian of $K$ to $1$.

We recall here an easy lemma for later use.

\begin{lem}[{\cite[Proposition~2.1]{Libcyclic}}]\label{Div}
Let $G_1$ (resp. $G_2$) be a fundamental group associated to the epimorphism $\phi_1: G_1 \xrightarrow[]{} \Z$ (resp. $\phi_2: G_2 \xrightarrow{} \Z$). Suppose we have a surjective map $\phi: G_1 \xrightarrow[]{} G_2$ as in the following diagram:
\[
\xymatrix @!0 @C=4pc @R=3pc { 
    G_1 \ar@{->>}[rr]^{\phi} \ar[rd]_{\phi_1} && G_2 \ar[ld]^{\phi_2} \\ & \Z }.
\]
Then $\Delta_{G_2}^{\phi_2}(t)\mid \Delta_{G_1}^{\phi_1}(t).$
\end{lem}

\begin{proof}

Since there is a surjection $\phi\colon G_1 \to G_2$, we can choose presentations for $G_1$ and $G_2$ of the form
\[
G_1 = \langle x_1, \dots, x_n \mid r_1, \dots, r_m\rangle, \quad G_2 = \langle x_1, \dots, x_n \mid r_1, \dots, r_{m+m'}\rangle.
\]
Since $\phi_2 \circ \phi = \phi_1$, the Fox matrix $M_2$ for $(G_2, \phi_2)$ is obtained by adding $m'$ columns to the Fox matrix $M_1$ of $(G_1,\phi_1)$, and therefore all $(n-1)$-minors of $M_1$ are also $(n-1)$-minors of $M_2$, hence the require divisibility follows.  

\end{proof}


\section{Alexander polynomials divisibility for symplectic curves}\label{S3}
In this section, we will prove the two parts of the divisibility theorem for symplectic curves. First, let us fix the context.

\subsection{Pencils and singular fibrations}
In this section we want to introduce braid monodromies of symplectic (i.e. $J$-holomorphic) curves, as defined by Kharlamov and Kulikov~\cite[Section~4]{KharlamovKulikov}.

In order to define it, we start by recalling that, similarly to the case of algebraic curves, one can define a singular fibration over $\C\P^1$ associated to a plane symplectic curve as follows. This follows from Gromov's foundational work on $J$-holomorphic curves~\cite{Gromov}.

Let $\sC$ be a symplectic curve in $(\CP, \omega_{\rm FS})$ and $J$ an $\omega_{\rm FS}$-compatible almost-complex structure on $\CP$ such that $\sC$ is $J$-holomorphic. Let $L$ be a $J$-holomorphic line and let $\sC^+=\sC \cup L$.
We are going to define a $J$-linear projection from a point $p\in L$ to $\CPone$. To do this, recall that for each $q\neq p$ in $\CP$ there exists a unique $J$-holomorphic line $L_{p,q}$ passing through $p$ and $q$~\cite{Gromov}. Choose an auxiliary $J$-holomorphic line\footnote{This is not necessary: one could also work with $\P(T_p\CP)$ instead, and it would be completely equivalent.}  $L'$ which we identify with $\CPone$, and define $\pi \colon \CP \setminus \{p\} \to L'\cong \CPone$ by sending $q$ to $L_{p,q} \cap L'$.

From now on, we will assume that $p \not\in \sC$. By positivity of intersections~\cite{McDuff-positivity}, the intersection between every $J$-holomorphic line and $\sC$ comprises $d$ points, counted with multiplicity. In fact, pre-composing the restriction $\pi|_{\sC}$ with the normalisation map $\tilde\sC \to \sC$ gives a branched cover of degree $d$, ramifying at points that correspond to points where either $\sC$ has a singular branch or $L_{p,q}$ is tangent to a non-singular branch of $\sC$. 

Up to changing the point $p$, we can assume that all singular points lie in distinct fibers of $\pi$.

We can now define the braid monodromy morphism exactly as in the complex case, by restricting to the preimage of the set of regular values of the projection $\pi$ from $p$. The total space of this restriction is a locally-trivial fibration with fibre $\C$ with $d$ points removed, and there is an associated monodromy map with values in the braid group $B_d$.

Furthermore, the braid monodromy of a curve determines the curve up to symplectic isotopy~\cite{KharlamovKulikov}. Contrarily to the complex case, however, to every braid monodromy one can associate a symplectic curve. 
Equivalently, one can think of suitable factorisations of the full twist (a braid whose closure is the generic link at infinity). That not all such factorisations correspond to complex curves was observed by Moishezon, who distinguished between \emph{geometric} and \emph{analytic} factorisations~\cite{Moi}.

\subsection{Divisibility at infinity}

We now prove the divisibility relation at infinity for Oka polynomials of symplectic curves, which readily implies (i) in Theorem~\ref{t:main}.

\begin{thm}\label{t:infty}
Let $\sC \subset \CP$ be a symplectic curve and $\sC^+ := \sC \cup L$ be $\sC$ with a line at infinity added and $K_\infty$ be the corresponding link at infinity. Let $\phi\co \pi_1(\CP\setminus\sC^+)\to \Z$ be a surjective homomorphism and denote with $\phi$ also the induced homomorphism $\pi_1(S^3\setminus K_\infty) \to \Z$. Then the Oka polynomial $\Delta^\phi_{\sC, L}(t)$ of $\sC$ with respect to $L$ and $\phi$ divides the Oka polynomial at infinity $\Delta^\phi_{K_\infty}(t)$.
\end{thm}
Note that the same proof holds for the Alexander polynomial $\Delta_{\sC,L}(t)$.\\
In what follows, we denote with $D_m$ a closed disc with $m$ small open disks removed from its interior. With an abuse of notation, we denote with $\del D_m$ just the outer boundary of $D_m$, which is a single copy of $S^1$.

\begin{proof}
Choose an almost-complex structure $J$ on $\CP$ for which the curve $\sC^+$ is $J$-holomorphic.
Choose a small tubular neighbourhood $\nu(L)$ of $L$ in $\CP$ such that the intersection of $\sC$ and $\del \nu(L)$ is transverse.
Call $Y_\infty := \del \nu(L) \setminus \nu{(\sC)}$.
Note that $Y_\infty$ is the complement of a small regular neighbourhood of $K_\infty$ in $S^3$.
Using the notation introduced in  subsection~\ref{Handle}, $Y_\infty=(\del \hD \times \hD_y) \setminus \nu(\sC) \cup (\hD \times \del \hD_y)$.
We will prove that $\pi_1(Y_\infty)$ surjects onto $\pi_1(\CP \setminus \sC^+)$.

For that, we will find a suitable handle decomposition of $\C\P^2 \setminus \nu(\sC^+)$ relative to $Y_\infty$. In order to do so, 
following \cite{sugawara2023handle}, we will show that $\CP \setminus \sC^+$ has a handle decomposition relative to $Y_\infty$ with handles of indices greater than or equal to 2.

Let $\pi_0 \colon (\hD \times \hD_y) \setminus \nu(\sC) \to \hD$ denote the restriction of the projection map to (the compactification of) the complement of $\sC^+$.
Let $\nu(\Gamma')= \bigcup_{i=0}^{s} U_i \cup \bigcup_{i=1}^{s} \nu(S_i)$, with $\nu(S_i)$ a small tubular neighbourhood of $S_i$ (see Figure~\ref{1l}). Notice that what differs from $\nu(\Gamma)$ defined in Section ~\ref{S2}, is the absence of $S_0$ the path connecting the basepoint to $\partial\mathbb{D}$.

We start by decomposing $\pi_0^{-1}(\nu(S_i)\cup U_i)$ for each $i = 1,\dots,s$ (see Figure \ref{2l}). Let $V_i\subset U_i$ be a small disk around $y_i$ such that $\pi_0^{-1}(v) \cong D_{d-(m_i-1)}$, for all $v \in V_i$.
Recall that on $L_{y_i}$ there is a unique point that does not intersect $\sC$ transversely, and $m_i$ is the multiplicity of intersection of $L_{y_i}$ and $\sC$ precisely at that point.
As in Section~\ref{S2}, we decompose the complement of $V_i$ in $U_i$ into two parts: $A_i$ and $B_i=\overline{U_i\setminus(A_i \cup V_i)}$.
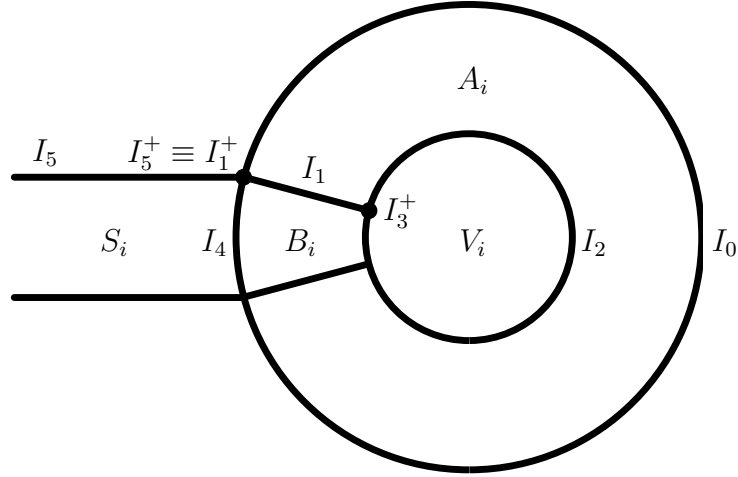
\begin{figure}
\labellist
\pinlabel $I_0$ at 207 67 
\pinlabel $I_2$ at 169 67
\pinlabel $I_4$ at 59 67
\pinlabel {$I_5^+ \equiv I_1^+$} at 50 93
\pinlabel $I_5$ at 10 93
\pinlabel {$I_3^+$} at 113 76
\pinlabel $I_1$ at 88 88
\pinlabel $A_i$ at 134 114
\pinlabel $B_i$ at 84 67
\pinlabel $V_i$ at 134 67
\pinlabel $S_i$ at 30 67
\endlabellist
\includegraphics[scale=1.3]{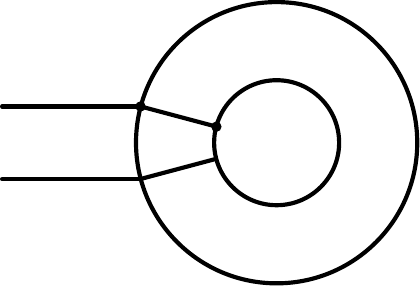}
 \caption{Attaching areas of $U_i$, $ i=1,\dots,s$.}\label{3l}
 \end{figure}

Using the decomposition of $\Gamma'$, we attach parts one by one. First, let us study the attachment of $\Tilde{A_i}=\pi_0^{-1}(A_i)$ to $Y_\infty$.
As noted by Sugawara~\cite[Section~3]{sugawara2023handle}, $\Tilde{A_i} \cong I_0\times I_1 \times D_d$, where $I_0$ and $I_1$ are intervals. We will denote by $I_1^-$ and $I_1^+$ respectively the interior and exterior boundary of $I_1$ such that $\del I_1 = \{I_1^-,I_1^+\}$ (see Figure~\ref{3l}). In the trivialisation above, the attaching area of $\Tilde{A_i}$ to $Y_\infty$ to $I_0 \times\{ I_1^+\} \times D_d \cup I_0\times I_1\times \del D_d  =X \cup Y$, and here $X$ and $Y$ intersect along $I_0\times\{ I_1^+\}\times \del D_d \subset Y$. One can notice that $Y$ is a collar neighbourhood of one component of $X$. Thus by Lemma~\ref{LemSug}, attaching $\Tilde{A_i}$ to $Y_\infty$ does not affect the diffeomorphism type.

Moreover, $\Tilde{V}_i=\pi_0^{-1}(V_i) \cong I_2\times I_3\times D_{d-m_i+1}$. In this trivialisation, the attaching area of $\Tilde{V_i}$ to $Y_\infty \cup \Tilde{A_i}$ is $I_2 \times \{I_3^+\} \times D_{d-m_i+1} \cup I_2 \times I_3 \times \del D_{d-m_i+1} $, and the two parts of this attaching area meet along $I_2 \times \{I_3^+\} \times \del D_{d-m_i+1}$ (see Figure \ref{3l}). Again by Lemma~\ref{LemSug}, this attachment does not affect the topology.
From~\cite[Theorem~3.4]{sugawara2023handle}, attaching $\pi\inv(B_i)$ corresponds to attaching $m_i-1$ 2-handles, and furthermore~\cite[Remark~3.7]{sugawara2023handle} asserts that we can view these 2-handle as being attached to the boundary at infinity.

Now one needs to attach $\widetilde{\nu(S_i)}=\pi_0^{-1}(\nu(S_i))=I_4 \times I_5 \times D_d$ to $Y_\infty \cup \Tilde{A_i} \cup \Tilde{V_i} \cup \Tilde{B_i}$ along $\del I_4 \times I_5 \times D_d \cup I_4 \times \{I_5^+\} \times D_d \cup I_4\times I_5 \times \del D_{d}=X \cup Y \cup Z$, where the three parts of the attaching area are glued along $\del I_4 \times \{I_5^+\}\times \del D_{d} \subset X \cup Z$. One can notice that  each of $X$ and $Z$ is a collar neighbourhood with respect to one component of $Y$. Hence attaching $\widetilde{\nu(S_i)}$ does not change the diffeomorphism type (see Figure~\ref{3l}).

Finally,  one would attach $\Tilde{U_0}= \pi_0^{-1}(U_0) \cong U_0 \times D_d$, where we view $D_d$ as obtained from $\del D_d \times I$ by attaching $d$ 2-dimensional 1-handles $H_1, \dots, H_d$. One can notice that $\del\Tilde{U_0}$ is divided into two parts. First, $\del (U_0 \times \del D_d \times I) =\del U_0 \times \del D_d \times I \cup U_0 \times \del D_d \times \del I $, which are glued along $U_0 \times \del D_d \times \{I^+\}$.

The second part corresponds to $\del(U_0 \times H_i)=\del(U_0 \times I_i \times I_i')=\del U_0 \times I_i\times I_i'\cup U_0\times \del I_i \times I_i' \cup U_0 \times  I_i \times \del I_i' $  whose attaching area is $\del U_0 \times I_i\times I_i' \cup U_0 \times  I_i \times \del I_i'$ for $i \in \{1,\dots,d\}$ corresponding to a 3-handle. Hence attaching $\Tilde{U_0}$ to  $Y_\infty \cup \bigcup_i(\Tilde{A_i} \cup \Tilde{V_i} \cup \Tilde{B_i} \cup \nu(S_i))$ is equivalent to attaching $d$ $3$-handles.
This means that we are only attaching $2$- and $3$-handles to $Y_{\infty}$ to obtain $\CP\setminus \sC^+$, so that $\pi_1(Y_\infty)$ surjects onto $\pi_1(\CP\setminus \sC^+)$, which implies the required divisibility relation by Lemma~\ref{Div}.
\end{proof}

We can now prove Corollary~\ref{c:primepowerdegreeirreducible} from the introduction, asserting that an irreducible symplectic curve $\sC$ of prime power degree have Alexander polynomial $1$ with respect to a generic line at infinity.

\begin{proof}[Proof of Corollary~\ref{c:primepowerdegreeirreducible}]
Call $\Delta_{\sC}$ the Alexander polynomial of $\sC$ with respect to a generic line at infinity.
Let $\deg \sC = p^r$, where $p$ is a prime and $r$ is a positive integer.

By the divisibility at infinity, Theorem~\ref{t:main}(i), and the choice of a generic line at infinity, $\Delta_{\sC}$ divides $\Delta_{T(p^r,p^r)}$, which, by Example~\ref{ex:deltahat} is $(1-t)(t^{p^r}-1)^{p^r-2}$.
In particular, $\Delta_{\sC}$ has roots which are $p^r$-th roots of 1.

Since $\sC$ is irreducible, $\CP\setminus (\sC\cup L)$ is a homology $S^1$.
By~\cite[Theorem~1]{Sum}, its $n$-fold cyclic cover has positive first Betti number if and only if $\Delta_{\sC}$ has a root which is an $n\th$ root of $1$.


Finally, let us consider the case of $p^r$-fold covers.
First, note that there is no ramification over $L$ in the associated $p^r$-fold \emph{branched} cyclic cover of $\CP$, so that the $p^r$-fold cover of $\CP$ ramifies only over $\sC$.
$H_1(\CP\setminus \sC) = \Z/p^r\Z$, and by the Goldschmidt lemma~\cite[Lemma~4.1 and the following corollary]{HsiangSzczarba}, its universal abelian cover has finite $H_1$.
Therefore, the associated branched cover, too, has finite $H_1$, and hence its first Betti number vanishes.

Combining these three observations, we conclude the proof of the corollary.
\end{proof}


\subsection{Local divisibility}

We now prove the local divisibility theorem for Oka and Alexander polynomials of symplectic curves. This readily implies (ii) in Theorems~\ref{t:main1} and ~\ref{t:main}.\\

First, we will fix some notations. Let $\mathcal{C}$ be a symplectic curve of $\ell$ irreducible components denoted by $\sC_1,\cdots,\sC_{\ell}$. Let $L$ be the line at infinity, and $\mathcal{C}^+=\mathcal{C}\cup L$. For notational convenience, we may denote $L$ by $\sC_0$ as well.\\ We denote by $\sC^{\rm aff}=\bigcup_{i=1}^{\ell} \sC^{\rm aff}_i$ the affine part of $\sC$ (i.e. $\sC^{\rm aff}=\sC \setminus \sC_{0})$) and by $s=s^{\rm aff}+s^{\infty}$ the number of singular points of $\sC^+$ splitted into singularities on the affine part of $\sC$ and the ones on $\sC_0$. We will denote by $s_j=s_j^{\rm aff}+s^{\infty}_j$ the number of singularities of $\sC_j$.\\
Let $(\sC^+)^{\rm ns}$ be the non singular part of $\sC^+$, $\sC^{\rm ns}$ (resp. $\sC^{\rm aff,ns}$) the non singular part of $\sC$ (resp. $\sC^{\rm aff}$), and $\sC_j^{\rm ns}$ (resp. $\sC_j^{\rm aff,ns}$) the non singular part of the $j$-th component $\sC_j$ (resp. $\sC_j^{\rm aff}$). 

\begin{thm} \label{bound}
With the previous notations, we have that:
\begin{itemize}
\item the Oka polynomial of $\sC$ relative to $L$, $\Delta^{\phi}_{\sC,L}(t)$, divides
\[
\prod_i \hat \Delta_{K_i}( t)\cdot\prod_{j=0}^{\ell}(1-t^{a_j})^{-\chi(\sC_j^{ns})+2};
\]
 \item if all $\sC_j$ are rational then  $\Delta^{\phi}_{\sC,L}( t)$ divides $\prod_i \hat \Delta_{K_i}( t)\cdot\prod_{j=0}^{\ell}(1-t^{a_j})^{-\chi(\sC_j^{ns})+1}$;
\item the Alexander polynomial  $\Delta_{\sC,L}( t)$ divides $\prod_i \hat \Delta_{K_i}( t)\cdot(1-t)^{b_1(\mathcal{C}^+)}$. 
\end{itemize}
Here $\hat \Delta_{K_i}(t)$ is the (twisted) Alexander polynomial of the link of the singularity $K_i$ of $\sC^+$ at a singular point $p_i$.
\end{thm}
First let us recall that a handle decomposition of a neighbourhood $N_{\sC^+}$ of  $\sC^+$ with 0, 1 and 2-handles was described in Section~\ref{S2}.
We call $J=J_0 \cup J_1 \cup \dots  \cup J_{\ell} \subset (S^1 \times S^2)^{\#b}$ the attaching circles of the 2-handles of this decomposition, where the component $J_0$ is the one that corresponds to the line at infinity $L$. 

To prove this theorem we proceed in two steps.

\subsubsection{Knot-theoretic preliminaries}\label{step1 symp}

We now dive in a detailed analysis of the behaviour of Oka polynomials with respect to connected sums and knotifications, and compute (certain) Oka polynomials of Borromean knots.

\begin{prop}\label{p:connectedsum}
For $j = 1,2$, let $M_j$ be a closed, oriented $3$-manifold, $K_j \subset M_j$ be a link with null-homologous components, and $\phi_j\co H_1(M_j\setminus K_j) \to \Z$ be a homomorphism. Suppose $(M, K)$ is the connected sum of $(M_1,K_1)$ and $(M_2,K_2)$ performed along components of $K_1$ and $K_2$ whose meridians $\mu_1$ and $\mu_2$ satisfy $\phi_1(\mu_1) = \phi_2(\mu_2) = \pm a$, for some $a > 0$. Let $\phi$ be the unique induced homomorphism $H_1(M\setminus K) \to \Z$. Then
\[
\Delta^\phi_K(t) \mid (t^a-1) \cdot \Delta^{\phi_1}_{K_1}(t) \Delta^{\phi_2}_{K_2}(t).
\]
If, additionally, $a = 1$ or $\im \phi_2 = a\Z$, then:
\[
\Delta^\phi_K(t) = \Delta^{\phi_1}_{K_1}(t) \Delta^{\phi_2}_{K_2}(t).
\]
\end{prop}

\begin{proof}
Call $X_1$, $X_2$, and $X$ the infinite cyclic covers of $M_1\setminus K_1$, $M_2\setminus K_2$, and $M\setminus K$ associated to $\phi_1$, $\phi_2$, and $\phi$, respectively.

Note that $M\setminus K$ can be obtained by gluing $M_1\setminus K_1$ and $M_2\setminus K_2$ along an annulus (the sphere along which we take the connected sum, pierced twice by $K$).
Since $a \neq 0$, this annulus lifts to $a$ copies of its universal cover, which is a disjoint union of $a$ copies of a contractible space.
Call $A$ a thickening of this infinite cyclic cover.
As we have just seen, $H_1(A) = 0$ and $H_0(A) \cong \Z^{\oplus a}$. In the basis given by its connected components, the action of $\Z$ on $H_0(A)$ is by cyclic permutation of the generators.
In particular, as a $\Z[t,t\inv]$-module, $H_0(A) = \Z[t,t\inv]/(t^a-1)$, so its order is $(t^a-1)$.

We now apply the Mayer--Vietoris theorem to $(X, X_1, X_2, A)$.
We have
\[
0 = H_1(A) \to H_1(X_1) \oplus H_1(X_2) \to H_1(X) \to H_0(A),
\]
and, since $\phi$ is compatible with $\phi_1$ and $\phi_2$, all maps are maps of $\Z[t,t\inv]$-modules.
It follows that the order of $H_1(X)$ divides the product of the orders of $H_1(X_1)\oplus H_1(X_2)$ and of $H_0(A)$, which proves the first statement.

Let us now prove the second statement. Note that the assumption $a=1$ is a special case of the assumption $\im \phi_2 = a\Z$, so we only need to prove the latter.
If $\im \phi_2 = a\Z$, we have $H_0(A) \cong H_0(M_2\setminus K_2) \cong \Z[t,t^{-1}]/(t^a-1)$ and the inclusion of $A$ into $M_2\setminus K_2$ induces an isomorphism. The Mayer--Vietoris sequence now tells us that we have an isomorphism of $\Z[t,t\inv]$-modules $H_1(X) \cong H_1(X_1) \oplus H_1(X_2)$.
\end{proof}

\begin{prop}\label{p:knotification}
Let $M'$ be a closed, oriented $3$-manifold, $K'\subset M'$ a link with null-homologous components, and $\phi'\co H_1(M'\setminus K') \to \Z$ be a homomorphism. Choose two components of $K'$ whose meridians $\mu_1$, $\mu_2$ satisfy $\phi'(\mu_1) = \phi'(\mu_2) = \pm a$ for some $a>0$, and let $(M,K)$ be the knotification of $(M',K')$ along these two components. Finally, let $\phi\co H_1(M\setminus K) \to \Z$ be any homomorphism extending $\phi'$. Then:
\[
\Delta^\phi_K(t) \mid (t^a-1) \cdot \Delta^{\phi'}_{K'}(t).
\]
If, additionally, $a = 1$ or $\im \phi = \im \phi' = a\Z$, then:
\[
\Delta^\phi_K(t) = (t^a-1)\cdot \Delta^{\phi'}_{K'}(t).
\]
\end{prop}

\begin{proof}
As in the proof of Proposition~\ref{p:connectedsum}, $M\setminus K$ can be cut along an annulus to obtain $M'\setminus K'$. As above, we then call $A$, $X$, and $X'$ the infinite cyclic covers of a thickening of this annulus, of $M\setminus K$, and of $M'\setminus K'$, respectively. We apply the Mayer--Vietoris long exact sequence to the quadruple $(X, X', A, X'\cap A)$. Note that $B := X' \cap A$ is just two copies of $A$, and that $H_0(B) = H_0(A) \oplus H_0(A)$, with the inclusion $\iota\co B \hookrightarrow A$ inducing the map $\iota_*\co H_0(A) \oplus H_0(A) \to H_0(A)$ that sums the two components. Of course, $\ker \iota_* \cong H_0(A)$ as $\Z[t,t\inv]$-modules.

Since $A$ is a union of $a$ contractible spaces, the Mayer--Vietoris sequence now reads:
\[
0 = H_1(B) \to H_1(X') \to H_1(X) \to H_0(B) \to H_0(X') \oplus H_0(A).
\]
Call $P$ the kernel of the map $H_0(B) \to H_0(X') \oplus H_0(A)$. We extract the exact sequence:
\[
0 \to H_1(X') \to H_1(X) \to P.
\]
Since $P$ is a submodule of $\ker \iota_*$, it is torsion and its order divides that of $\ker \iota_*$, which in turn is $(t^a-1)$. The required divisibility is now an easy consequence.

If $a = 1$ or $\im \phi = \im \phi' = a\Z$, $P$ is exactly $\ker \iota_* \cong \Z[t,t\inv]/(t^a-1)$, so the required equality follows.
\end{proof}

\begin{prop}\label{p:borromean}
Let $B_g \subset (S^1\times S^2)^{\#2g}$ be the $g\th$ Borromean knot and $\phi\co H_1((S^1\times S^2)^{\#2g} \setminus B_g) \to \Z$. Call $a$ the positive generator of the image of $\phi$, and suppose that the meridian of $B_g$ is sent to $\pm a$. Then
\[
\Delta^\phi_{B_g}(t) = (t^a-1)^{2g}.
\]
\end{prop}

\begin{proof}
Observe that $B_1$ is the knotification of the 2-component link $J_K := S^1\times \{x,-y\} \subset S^1 \times S^2$. (Here $-y$ simply means that we reverse the orientation of the component $S^1\times \{y\}$.) We have $\pi_1(S^1\times S^2 \setminus J_K) = \pi_1(S^1)\times \pi_1(S^2\setminus \{x,-y\}) \cong \Z \oplus \Z$, and each of the two meridians generates the second summand.

One then easily\footnote{For example, by Fox calculus, since $S^1\times S^2 \setminus J_K \cong S^1\times S^1$, and the two meridians of $J_K$ correspond to the same primitive generator of $\pi_1(S^1\times S^1)$.} computes that $\Delta_{S^1\times S^2, J_K}(t) = t^a-1$, and the required equality follows from the second statement in Proposition~\ref{p:knotification}.

For the general case, we now apply the second statement in Proposition~\ref{p:connectedsum}.
\end{proof}

\subsubsection{The proof of Theorem~\ref{bound}}

We start with a lemma.

\begin{lemma}\label{l:ontoexterior} Let $Y_{\sC^+}$ be the boundary of a neighbourhood of the symplectic curve $\sC^+$.\\
The inclusion of $Y_{\sC^+}$ into $\P^2\setminus\sC^+$ induces a surjection:
\[
\pi_1(Y_{\sC^+}) \twoheadrightarrow \pi_1(\P^2\setminus\sC^+).
\]
\end{lemma}

\begin{proof}
This is straightforward by \cite{sugawara2023handle}, since $\P^2\setminus\sC^+ $ admits a handle decomposition with handles of indices 0, 1 and 2. 
\end{proof}

We are almost ready to prove Theorem~\ref{bound}, namely the local divisibility for Oka polynomials. We briefly recap the situation we are in.

Recall from Section~\ref{ss:sympl_curves} that we gave a surgery description of the boundary manifold $Y_{\sC^+}$. More precisely, we found an explicit link $J \subset (S^1\times S^2)^{\#N}$, obtained from the links of the singularities of $\sC^+$ via as a sequence of coloured connected sums, coloured knotifications, and connected sums with Borromean knots.
It is useful to remind ourselves how these Borromean knots come up: we need to choose a handle decomposition of the surface, relative to a ball containing all the singularities, and we want this handle decomposition to be minimal, so that it contains $g_j$ 1-handles for the $j\th$ component of $\sC$.
The Borromean knots arise as the boundary of the 2-handles of this handle decomposition. Moreover, we can identify the connected summands of the manifold $(S^1\times S^2)^{\#N}$ coming from the Borromean knots with the 1-handles of these decompositions.

We state an easy lemma before getting to the proof of Theorem~\ref{bound}.

\begin{lemma}\label{l:chooseborromean}
Let $D_j \subset \sD$ be an irreducible component of a symplectic curve $\sD \subset \CP$ and $\phi\co H_1(\CP\setminus \sD) \to \Z$ be a homomorphism, and call $a_j = \phi(\mu_j)$, the image of a meridian of $D_j$. We can choose a handle decomposition of $D_j$ as above, so that at most one of (the push-offs of) the corresponding generators $\alpha \in H_1(D_j)$ satisfies $\phi(\alpha)\not\in a_j\Z$
\end{lemma}

\begin{proof}
This is much easier to prove than it was to state. Call $a$ the generator of the image of $\phi$, restricted to the push-offs of $D_j$. (Technically, this is well-defined only up to multiples of $a_j$, but this is ok for our purposes.) B\'ezout's theorem guarantees for us that we can do handle-slides among the 1-handles of $D_j$ (which corresponds to sums and differences in $H_1(D_j)$) until we have a 1-handle $\alpha$ for which $\phi(\alpha) = a$. Sliding every other 1-handle over $\alpha$, if necessary, we can now find the desired handle decomposition.
\end{proof}

\begin{proof}[Proof of Theorem~\ref{bound}]
Fix a homomorphism $\phi\co H_1(\CP \setminus \sC^+) \to \Z$, not sending any meridian to $0$.
For the sake of making the notation lighter, let us write $Y$ instead of $Y_{\sC^+}$ for the boundary of the neighbourhood of $\sC^+$ and $E_J$ instead of $(S^1\times S^2)^{\#N} \setminus J$.

Note that $\phi$ induces homomorphisms $\phi_Y \co H_1(Y) \to \Z$ and $\phi_J: H_1(E_J) \to \Z$ which commute with the inclusion (as well as with the inclusion into $\CP \setminus \sC^+$).

By Lemma~\ref{l:ontoexterior} we have:
\[
\Delta^{\phi}_{{\sC},L} \mid \Delta^{\phi_Y}_{Y}.
\]
In order to prove the divisibility, it will suffice to compute $\Delta^{\phi_Y}_{Y}$.

We will actually estimate $\Delta^{\phi_Y}_{Y}$, first by estimating $\Delta^{\phi_J}_J$ and then by understanding the relationship between the two Alexander polynomials. (Note that $\pi_1(Y)$ is a quotient of $\pi_1(E_J)$, which implies that $\Delta^{\phi_Y}_{Y}$ divides $\Delta^{\phi_J}_J$.)

Recall that $J$ is obtained from (any) coloured connected sum of the links of $\sC^+$ by doing as many coloured knotifications as possible, and then by taking a connected sum with as many Borromean knots as the components of $\sC$ with positive genus.

There are two things to check: the number of coloured connected sums and coloured knotifications we do for each colour (i.e. for each component of $\sC^+$), and that the assumptions of Proposition~\ref{p:borromean} are satisfied. We start with the latter. Call $J_c$ the link we obtain after doing all the coloured knotifications. 

Fix a component $\sC_j$ of $\sC^+$. The ``Borromean'' part of $(S^1\times S^2)^{\#N}$ corresponding to $\sC_j$ comes from a handle decomposition of $\sC_j$ as above. Thanks to Lemma~\ref{l:chooseborromean}, we can choose all but the first of these components to satisfy $\phi_B \co H_1((S^1\times S^2)^{\#2} \setminus B_1) \to \Z$ to have $\im \phi_B = a_j\Z$ (as usual, $a_j$ is the image of the meridian of $\sC_j$ via $\phi$). In particular, we can apply Propositions~\ref{p:borromean} and~\ref{p:connectedsum} to conclude that $\Delta^{\phi_J}_J(t)$ divides $\Delta^{\phi_{J_c}}_{J_c}(t)\cdot \prod_j (1-t^{a_j})^{2g_j+1}$. \\
Notice that, from one hand, if all $a_j=1$ then by Proposition~\ref{p:connectedsum},  $\Delta^{\phi_J}_J(t)$ divides $\Delta^{\phi_{J_c}}_{J_c}(t)\cdot \prod_j (1-t^{a_j})^{2g_j}$. From other hand, in case all $\mathcal{C}_j$ are rational curves (i.e. of genus 0) then $\Delta^{\phi_J}_J(t) \ \mid \ \Delta^{\phi_{J_c}}_{J_c}(t)$. 

Similarly, applying Propositions~\ref{p:connectedsum} and~\ref{p:knotification}, we obtain that $\Delta^{\phi_{J_c}}_{J_c}(t)$ divides $\prod \Delta^{\phi_{K_i}}_{K_i}(t) \cdot \prod (1-t^{a_j})^{h_j}$, where $h_j$ is the number of paths connecting the singularities of $\sC_j$ that are needed to cut $\sC_j$ open into a surface of genus $g_j$ with one puncture. (In other words, if we let $\beta_{ij}$ be the number of branches of $\sC_j$ at the singular point $p_i$, we have $h_j = \big(\sum_i \beta_{ij}\big) -1$ for each $j$).``
Moreover, if all $a_j=1$ then by Proposition ~\ref{p:connectedsum}, $ \Delta^{\phi_{J_c}}_{J_c}(t)$ divides $\prod \Delta^{\phi_{K_i}}_{K_i}(t) \cdot \prod (1-t^{a_j})^{h_j-(s-1)}$.\\
Thererfore, we have that  $\Delta^{\phi}_{\sC,L}\ \mid \  \Delta^{\phi_{J}}_{J}(t) \ \mid \ \prod \Delta^{\phi_{K_i}}_{K_i}(t) \cdot \prod (1-t^{a_j})^{2g_j+1+h_j}$. One can then conclude using the fact that $\chi(\sC^{\rm ns})=\chi(\sC)-s$ and $\chi(\sC_j^{\rm ns})=\chi(\sC_j)-s_j$ so that the exponent of $(1-t^{a_j})$ becomes $2g_j+1+h_j=2g_j+1+(\sum_{p\in sing(\sC_j)}\beta(p))-1=s_j-(2-2g_j-\sum_{p\in sing(\sC_j)}(\beta(p)-1))+2=-\chi(\sC^{ns}_j)+2$. \\If all $\sC_j$ are rational, then $\Delta^{\phi}_{\sC,L}\  \mid \ \prod \Delta^{\phi_{K_i}}_{K_i}(t) \cdot \prod (1-t^{a_j})^{-\chi(\sC^{ns}_j)+1}$. Moreover, for all $a_j=1$,  $\Delta^{\phi}_{\sC,L}\  \mid \ \prod \Delta^{\phi_{K_i}}_{K_i}(t) \cdot \prod_{j=0}^{\ell} (1-t)^{-\chi((\sC^+)^{ns}+\ell-s+2}$ with $-\chi((\sC^+)^{ns}+\ell-s+2=b_1(\sC^+)$.
\end{proof}
The factors $(1-t)$ in $\Delta_{\sC,L}$ play a special role, and we want to give bounds on their multiplicity. When $\Delta_{\sC,L}$ vanishes, then $\Delta_{\sC,L}$ is divisible by arbitrarily large powers of $(1-t)$ and we define the multiplicity of $(1-t)$ as a factor of $\Delta_{\sC,L}$ to be $\infty$, whereas whenever $\Delta_{\sC,L}$ is non-zero the multiplicity of $(1-t)$ is finite.

\begin{prop}\label{l-1}
Let $\sC \subset \mathbb{C}\mathbb{P}^2$ be a curve of degree $d$ and $\ell$ irreducible components. Choose a line at infinity $L$, and let $m$ be the multiplicity of $(t-1)$ as a factor of the Alexander polynomial $\Delta_{\sC,L}$ (which could be $\infty$). Then:
\[
\ell - 1 \leq m.
\]
Moreover, if $L$ is transverse to $\sC$, then $\Delta_{\sC,L}$ is non-zero and
\[
m \leq d-1.
\]
Finally, if $\sC$ is irreducible, then $m = 0$.
\end{prop}

\begin{proof}
We start by proving the lower bound on $m$. If $\Delta_{\sC,L} \equiv 0$, then $m = \infty$, so the inequality is trivially satisfied. Suppose therefore that $\Delta_{\sC,L} \neq 0$, or equivalently that $H_1(\tilde X; \Q)$ is a torsion $\Q[t,t\inv]$-module. (As above, $\tilde X$ is the infinite cyclic cover of $X=\CP\setminus (\sC \cup L)$ associated to the morphism that sends each meridian of $\sC$ to $1$, and $t$ is the action of the generator of the deck transformation group of the cover $\tilde X \to X$.)

Let us look at the long exact sequence in homology associated to the short exact sequence of complexes:
\[
0 \to C_*(\tilde X;\Q) \stackrel{1-t}{\longrightarrow} C_*(\tilde X;\Q) \stackrel{\pi_*}{\longrightarrow} C_*(X) \to 0,
\]
where $\pi \co \tilde X \to X$ is the covering map.

In this proof, we will be using $\Q$-coefficients in homology throughout, and we will not make it explicit in the notation. Since we are assuming that $H_1(\tilde X)$ is torsion, we can write it as
\[
H_1(\tilde X) = T \oplus T_1,
\]
where $T_1 = \bigoplus_{j=1}^p \Q[t^{\pm 1}]/(1-t)^{a_j}$ (with $a_j \ge 1$ for each $j$) and the order of $T$ is coprime with $(1-t)$. Note that the multiplicity of $(1-t)$ in $\Delta_{\sC,L}$ is simply $m = \sum_{j = 1}^p a_j \ge p$.
The long exact sequence now gives us:
\[
\xymatrix{
H_1(\tilde X) \ar[r]^{1-t}\ar[d]^{\cong} & H_1(\tilde X) \ar[r] \ar[d]^{\cong} & H_1(X) \ar[d]^{\cong} \ar[r]^\alpha & H_0(\tilde X) \ar[d]^{\cong}\ar[r]^{1-t} & H_0(\tilde X) \ar[r]^\beta \ar[d]^{\cong} & H_0(X)\ar[d]^{\cong}\\
T\oplus T_1 \ar[r] & T \oplus T_1 \ar[r] & \Q^\ell \ar[r]^\alpha & \Q \ar[r] & \Q \ar[r] & \Q
}
\]
The map $\beta$, induced by $\pi$, is obviously an isomorphism, so the map $\alpha$ is onto, and we can now obtain another exact sequence
\[
\xymatrix{
H_1(\tilde X) \ar[r]^{1-t}\ar[d]^{\cong} & H_1(\tilde X) \ar[r] \ar[d]^{\cong} & \ker \alpha \ar[r]\ar[d]^{\cong}& 0\\
T\oplus T_1 \ar[r] & T \oplus T_1 \ar[r] & \Q^{\ell-1}\ar[r]& 0.
}
\]
Now, the map $(1-t)$ is an isomorphism on $T$ and has cokernel $\Q^p$ on $T_1$, so $p = \ell-1$, and in particular
\[
\ell - 1 = p \le m = \sum_{j=1}^p a_j,
\]
as required.

We now turn to the upper bound, when $L$ is chosen to be generic. In this case, the link at infinity of $\sC$ with respect to $L$ is the (generalised) Hopf link $K_{\infty} =T(d,d) \subset S^3$ with $d$ components, for which it is known that the Alexander polynomial is $\Delta_{T(d,d)} = (t-1)(t^d-1)^{d-2}$ (see~\cite[Section~10]{Milnor-hypersurfaces}). From Theorem~\ref{t:infty} it follows that $\Delta_{\sC,L}$ is also non-zero and that the multiplicity of $(1-t)$ in $\Delta_{\sC,L}$ is bounded from above by $d-1$, as desired.

Finally, if $\sC$ is irreducible, then $\ker\alpha = 0$, so $T_1 = 0$, which in turn implies that $\Delta_{\sC,L}$ is coprime with $1-t$.
\end{proof}

\begin{rk}\label{r:vanishingAlex}
The second statement in Proposition~\ref{l-1} holds whenever $\Delta_{K_\infty}$ does not identically vanish, and even under the more general assumption that $H_1(\tilde X)$ is a torsion $\Q[t,t\inv]$-module (so that the associated Alexander polynomial is not identically 0).

One can easily construct examples when $H_1(\tilde X)$ is non-torsion. For instance, consider the curve $\sC = \bigcup_{k=1}^d\{x = kz\}$, and choose the line at infinity to be $\{z = 0\}$. In this case, $\CP\setminus \sC^+$ is diffeomorphic to $\C \times (\C \setminus \{1,\dots,k\}) \subset \C^2$, where the meridians of the components of $\sC$ can be chosen to be curves in the plane $\{0\} \times \C$. Since this is a 4-dimensional 1-handlebody, the action of the group ring of the deck transformation group on the infinite cyclic cover $\tilde X$ is free, so the $H_1(\tilde X;\Q)$ is a (non-trivial) free $\Q[t,t\inv]$-module, and hence the Alexander polynomial of $\sC$ (with respect to the line $\{z=0\}$) vanishes identically.
\end{rk}

\begin{cor}
For a symplectic line arrangement $\sC\subset \C\P^2$ of degree $\ell$ (that is, $\sC$ is a union of $\ell$ symplectic lines in $\C\P^2)$ and $L$ is a generic line at infinity, we have that $\Delta_{\sC,L}(t) \mid \prod_i \hat \Delta_{K_i}(t)$.
\end{cor}

\begin{proof}
By Proposition~\ref{l-1}, $\Delta_{\sC,L}(t) = (1-t)^{\ell-1}\cdot \Delta'_{\sC,L}(t)$, where $\gcd(\Delta'_{\sC,L}(t),1-t)=1$. By \cite[Section 10]{Milnor-hypersurfaces}, $\hat \Delta_{K_i}(t) = (1-t)^{\ell_i-1} \cdot \hat \Delta'_{K_i}(t)$, where $\ell_i$ the number multiplicity of the $i^{\rm th}$ singularity of $\sC^+$ (which is also the number of branches, since lines can only meet transversely) and $\gcd(\hat \Delta'_{K_i}(t),1-t)=1$ for each $i$. Therefore by Theorem~\ref{bound}, we have that $\Delta'_{\sC,L}(t)\mid \hat \Delta'_{K_i}(t)$  and we conclude using the combinatorial inequality $\ell-1 \leq \sum_i \ell_i -1$.
\end{proof}
\appendix
\section{Alexander polynomials divisibility for plane algebraic curves, after Libgober} \label{appendix}
Based on previous sections, we propose a brief twist of Libgober's proofs of the divisibility theorem for algebraic plane curves  that is proved in \cite{Libcyclic}, in the case where the line at infinity is chosen generically. 

\begin{thm}[\cite{Libcyclic}]\label{thmDiv}
Let $\sC \subset \CP$ be a complex curve and $L$ be a line that is transverse to $\sC$. Then, in the notation of Theorem~\ref{t:main}:
\begin{itemize}\itemsep 4pt
\item[(i)] $\Delta_{\sC,L}(t) \mid \Delta_{K_\infty}(t)$;
\item[(ii)] $\Delta_{\sC,L}(t) \mid \prod_{i} \hat \Delta_{K_{i}} (t)$.
\end{itemize}
\end{thm}

\subsection{Sketch of a new proof of Theorem \ref{thmDiv}:}
\begin{itemize}\itemsep 4pt
\item[(i)] Let $\mathcal{C}$ be defined by the polynomial $F$, i.e. $\mathcal{C} = \{F = 0\}$.
Suppose first that the line at infinity is chosen to be generic, and is defined by the (linear) polynomial $G$. Consider the following map, $p$, defined on $\CP \setminus (L \cap \mathcal{C})$:
\[
p\colon \CP \setminus (L \cap \mathcal{C}) \to \CPone, \quad p\colon (x:y:z) \mapsto (F(x:y:z) : G(x:y:z)^d).
\]
(In other words, we are looking at the pencil of degree-$d$ curves generated by $\mathcal{C}$ and the line $L$, counted with multiplicity $d$).
Since $L$ is chosen to be generic, $p$ is a Lefschetz fibration away from $\mathcal{C}$ and $L$. In particular, $\CP\setminus (\mathcal{C}\cup L)$ is the total space of a Lefschetz fibration over $\C^*$. It follows that there exists $R \gg 0$ such that $\CP\setminus (\mathcal{C}\cup L)$ is obtained from $p\inv(\{|z| \ge R\})$ by attaching one 2-handle for each singularity of $p$, and so its fundamental group is a quotient of the fundamental group of $p\inv(\{|z| \ge R\})$. The latter space retracts on the complement of the link at infinity (by definition of link at infinity), so we get a surjection of fundamental groups:
\[
\pi_1(S^3 \setminus K_\infty) \to \pi_1(\CP \setminus (\mathcal{C} \cup L)),
\]
which is easily seen to respect meridians. In particular, by Lemma \ref{Div} we get the desired divisibility.


\item[(ii)] For $\sC^+=\sC \cup L$, we consider on one hand the following divisbility:
\[\Delta_{\pi_1(Y_{\sC^+})}(t) \mid \Delta(\pi_1((S^1\times S^2)^{\#b_1({\sC^+})} \setminus J)) \mid
(1-t)^{b_1({\sC^+})}\cdot\prod_{i}\hat \Delta_{\pi_1(S^3\setminus K_i)}(t)\]
with $J$ the attaching link defined in Section~\ref{step1 symp}.

On the other hand, since  $\CP \setminus \sC^+$ is a $2$-dimensional Stein manifold, it is then the interior of a handlebody such that each handle has index $\leq 2$ (see~\cite{hdle} or more explicitly~\cite{sugawara2023handle}). We have that
\[
\xymatrix{ \pi_1(Y_{\sC^+}) \ar@{->}[r] &\pi_1(\CP\setminus\sC^+)\ar@{->}[r]& 1}.
\]
By Lemma~\ref{Div},
\[
\Delta_{\sC,L}(t) \mid \Delta_{\pi_1(Y_{\sC^+})}(t) \mid (1-t)^{{b_1({\sC^+})}}\cdot\prod_{i}\hat \Delta_{K_i}(t).
\]

If $\mathcal{C}$ is irreducible, then the map  $H_1(\widetilde{\CP\setminus \mathcal{C}^+}) \xrightarrow[]{1-t} H_1(\widetilde{\CP\setminus \mathcal{C}^+})$ is onto, hence $\Delta_{\sC,L}(t)$ and $(1-t)$ are coprime. Therefore, $\Delta_{\sC,L}(t) \mid \prod_{i}\hat \Delta_{K_i}(t)$.
Otherwise, if $L$ is generic, by  \cite[Lemma 21]{Oka}, we have that the factor $(1-t)$ has multiplicity exactly $\ell-1$ in the global Alexander polynomial of algebraic curves and $(\ell_i -1)$ in each local Alexander polynomial $\hat \Delta_{K_i}(t)$ such that $\ell_i$ denotes the number of local branches.

Let us write $\Delta_{\sC,L}(t) = (1-t)^{\ell -1} \Delta'_{\sC,L}(t)$ with $\Delta'_{\sC,L}(1) \neq 0$. Analogously, we write $\hat \Delta_{K_i}(t) = (1-t)^{\ell_i -1} \hat \Delta'_{K_i}(t)$, with $\Delta'_{K_i}(1) \neq 0$. Hence  the following divisibility
\[
(1-t)^{\ell -1} \Delta'_{\sC,L}(t)\mid (1-t)^{b_1({\sC^+})}(1-t)^{ \sum_i(\ell_i-1)}\prod_i\hat \Delta'_{K_i}(t)
\]
implies that  $\Delta'_{\sC,L}(t) \mid \prod\hat \Delta'_{K_i}(t)$. And since $\ell-1 \leq \sum_i (\ell_i-1)$, we conclude that  $\Delta_{\sC,L}(t) \mid \prod_{i}\hat \Delta_{K_i}(t)$.
\end{itemize}
\bibliography{divisibility}
\bibliographystyle{alpha}

\end{document}